\RequirePackage{fix-cm}
\documentclass{amsart} 

\usepackage{graphicx}
\usepackage{mathptmx}      
\usepackage{relsize}
\usepackage{hyperref}
\usepackage{latexsym,enumerate,
amsmath,amssymb,
amscd,enumerate,engord,amsfonts,graphics}
\usepackage[all]{xy}

\usepackage{float}

\usepackage{mathrsfs}

\usepackage{subcaption}
\usepackage{verbatim}
\usepackage{extarrows}

\theoremstyle{plain}
\newtheorem{proposition}[subsubsection]{Proposition}
\newtheorem{theorem}[subsubsection]{Theorem}
\newtheorem{lemma}[subsubsection]{Lemma}
\newtheorem{corollary}[subsubsection]{Corollary}
\newtheorem*{thm*}{Theorem}

\theoremstyle{remark}
\newtheorem{remark}[subsubsection]{Remark}
\newtheorem*{rem*}{Remark}

\theoremstyle{definition}
\newtheorem{definition}[subsubsection]{Definition}
\newtheorem{example}[subsubsection]{Example}

\newcommand{\image}{\ensuremath{\bar{N}}}
\renewcommand{\phi}{\varphi}
\newcommand{\A}{\ensuremath{\mathcal A}}

\newcommand{\D}{\ensuremath{\mathcal D}}

\newcommand{\X}{\ensuremath{\mathscr X}} 

\renewcommand{\P}{\ensuremath{\mathbb P}}
\newcommand{\PPP}{\ensuremath{\Pi}}
\newcommand{\PP}{\ensuremath{\mathcal {P}}}
\newcommand{\CP}{\ensuremath{\mathbb C}\P}
\newcommand{\C}{\ensuremath{\mathbb C}}
\renewcommand{\L}{\ensuremath{\mathscr {L}}}
\newcommand{\LL}{\ensuremath{\mathbf{L}}}

\newcommand{\R}{\ensuremath{\mathbb R}}

\newcommand{\Z}{\ensuremath{\mathbb Z}}
\newcommand{\Q}{\ensuremath{\mathbb Q}}
\newcommand{\B}{\ensuremath{\mathcal B}}
\renewcommand{\S}{\ensuremath{\mathcal S}}

\newcommand{\sS}{\ensuremath{\mathscr S}}
\newcommand{\F}{\ensuremath{\mathcal F}}

\newcommand{\KK}{\ensuremath{\mathsf K}}
\newcommand{\LLL}{\ensuremath{\mathsf L}}

\newcommand{\Fl}{\operatorname{\sf Flag}}
\newcommand{\id}{\operatorname{id}}
\newcommand{\im}{\operatorname{im}}

\newcommand{\Aut}{\operatorname{Aut}}
\newcommand{\Lie}{\operatorname{Lie}}

\newcommand{\supp}{\operatorname{supp}}

\newcommand{\rk}{\operatorname{rk}}
\newcommand{\ab}{\ensuremath{\operatorname{ab}}}
\newcommand{\lk}{\ensuremath{\operatorname{lk}}}

\newcommand{\codim}{\operatorname{codim}}
\newcommand{\coker}{\operatorname{coker}}

\renewcommand{\k}{\ensuremath{\Bbbk}}
\renewcommand{\bar}{\overline}

\renewcommand{\to}{\longrightarrow}

\newcommand{\bigast}{\mathop{\scalebox{1.4}{\raisebox{-0.2ex}{$\ast$}}}}%

\begin{document}

\title[Discriminantal bundles, arrangement groups, and subdirect products]{Discriminantal bundles, arrangement groups, and subdirect products of free groups}        

\author[D. Cohen]{Daniel C. Cohen}
\address{Department of Mathematics, Louisiana State University, Baton Rouge, Louisiana 70803}
\email{\href{mailto:cohen@math.lsu.edu}{cohen@math.lsu.edu}}
\urladdr{\href{http://www.math.lsu.edu/~cohen/}
{www.math.lsu.edu/\char'176cohen}}

\author[M. Falk]{Michael J. Falk}
\address{Department of Mathematics and Statistics, Northern Arizona University, Flagstaff, Arizona 86011}
\email{\href{mailto:michael.falk@nau.edu}{michael.falk@nau.edu}}
\urladdr{\href{http://www.cefns.nau.edu/~falk/}
{www.cefns.nau.edu/\char'176falk}}

\author[R. Randell]{Richard C. Randell}
\address{Department of Mathematics,University of Iowa, Iowa City, Iowa 52242}
\email{\href{mailto:randell@math.uiowa.edu}{randell@math.uiowa.edu}}

\dedicatory{
in memory of Bill Arvola and \c{S}tefan Papadima
}

\begin{abstract}
We construct bundles $E_k(\A,\F) \to M$ over the complement $M$ of a complex hyperplane arrangement \A, depending on an integer $k \geq 1$ and a set $\F=\{f_1, \ldots, f_\mu\}$ of continuous functions $f_i \colon M \to \C$ whose differences are nonzero on $M$, generalizing the configuration space bundles arising in the Lawrence-Krammer-Bigelow representation of the pure braid group. We display such families \F\ for rank two arrangements, reflection arrangements of types $A_\ell$, $B_\ell$, $D_\ell$, $F_4$, and for arrangements supporting multinet structures with three classes, with the resulting bundles having nontrivial monodromy around each hyperplane. The construction extends to arbitrary arrangements by pulling back these bundles along products of inclusions arising from subarrangements of these types.

We then consider the faithfulness of the resulting representations of 
the arrangement group $\pi_1(M)$. 
We describe the kernel of the product $\rho_\X \colon G \to \prod_{S \in \X} G_S$ of homomorphisms of a finitely-generated group $G$ onto quotient groups $G_S$ determined by a family \X\ of subsets of a fixed set of generators of $G$, extending a result of T.~Stanford about Brunnian braids.  When the projections $G \to G_S$ split in a compatible way, we show the image of $\rho_\X$ is normal with free abelian quotient, and identify the cohomological finiteness type of $G$. These results apply to some well-studied arrangements, implying several qualitative and residual properties of $\pi_1(M)$, including an alternate proof of a result of Artal, Cogolludo, and Matei on arrangement groups and Bestvina-Brady groups, and a dichotomy for a decomposable arrangement \A: either $\pi_1(M)$ has a conjugation-free presentation or it is not residually nilpotent. 
\end{abstract}


\maketitle

\setcounter{tocdepth}{2}
\tableofcontents

\section{Introduction}
Let $\A = \{H_1,\ldots,H_n\}$ be an arrangement of affine hyperplanes in $\C^{\ell}$, with complement $M=M(\A)=\C^\ell - \bigcup_{i=1}^n H_i$. Suppose $\F=\{f_1, \ldots, f_\mu\}$ is a set of complex-valued, continuous functions on $M$ whose pairwise differences are nowhere-zero on $M$. Then the function $f=(f_1, \ldots, f_\mu) \colon M \to \C^\mu$ maps $M$ into the complement $PB_\mu=M(\A_\mu)$ of the braid arrangement $\A_\mu$ consisting of the diagonal hyperplanes $z_i=z_j$, $1 \leq i<j \leq \mu$, with complement  the configuration space of $\mu$ distinct ordered points in $\C$. The fundamental group of $M(\A_\mu)$ is isomorphic to the pure braid group on $\mu$ strands $P_\mu$. The projection $\C^{\mu + k} \to \C^\mu$ onto the first $\mu$ coordinates restricts to a fiber-bundle projection $M(\A_{\mu+k}) \to M(\A_\mu)$, with fiber the ordered configuration space of $k$ points in a plane with $\mu$ punctures \cite{FN,Bi75}, realized as the complement of the affine discriminantal arrangement $\A_{\mu, k}$ in $\C^k$ \cite{SV}. The induced map $P_{\mu+k} \to P_\mu$ corresponds to deletion of the last $k$ strands. 
When $k=1$, the monodromy of this bundle is the restriction of the Artin representation to $P_\mu$. 
The $k=2$ case plays a role in the description of the Lawrence-Krammer-Bigelow (LKB) representation of the full braid group in \cite{PP02,Pa09}. 

A collection of functions \F\ as above yields a pullback diagram of bundles:
\[
\xymatrix@1{
E_k(\A,\F)\ar[r]\ar[d] & M(\A_{\mu+k}) \ar[d] \\
M\ar[r]^-{f}& M(\A_\mu). 
} 
\]
We call the bundle $E_k(\A,\F) \to M$ a {\em discriminantal bundle}. Its fiber is again the complement of the affine discriminantal arrangement $\A_{\mu, k}$ in $\C^k$. The set of functions \F\ is called a {\em generating set} for the discriminantal bundle. We establish general properties of discriminantal bundles in Section 2.1.  
For example, with a generating set of the form $\F=\{a_iz+b_i\bar{z}\}_{i=1}^n$, $a_i,b_i\in \C$, one can realize the complement of any arrangement of $2$-planes in $\R^4=\C^2$ as the total space of a discriminantal bundle  $E_1(\A,\F) \to \C^*$.

If \F\ consists of linear functions, the bundle $E_1(\A,\F) \to M$ is a strictly linear fibration over $M$ (with punctured plane fiber) as in the definition of fiber-type arrangements \cite{FR2}. In this case our construction coincides with the root map construction of \cite{CS2,Coh01}, used to produce the braid monodromy presentation of the fundamental group of the total space.  The fact that bundles involving strictly linearly fibered arrangements 
may be realized as pullbacks of configuration space bundles was established in \cite{Coh01}.  This was used to show that fundamental groups of complements of  fiber-type arrangements are linear in \cite{CCP07}.

The spaces $E_k(\A,\F)$ initially arose as spaces of arrangements with fixed combinatorial type, with the maps $E_k(\A,\F)\to M$ giving parametrizations -- see Remark~\ref{rem:r2}. The description in terms of generating sets allows one to produce representations of $\pi_1(M)$ directly from representations of the (pure) braid group, including the LKB representation, by pulling back along the homomorphism $f_* \colon \pi_1(M) \to P_\mu$. For the resulting representation to be faithful it is necessary that the monodromy around each hyperplane $H \in \A$ be nontrivial; in this case we say \F\ is {\em supported on} \A. This condition imposes strong restrictions on the underlying arrangement.  In Section 2.1 we exhibit generating sets supported on reflection arrangements of types $A_\ell$, $B_\ell$, $D_\ell$, $F_4$, and in Section 2.2 we construct generating sets supported on arbitrary rank two arrangements, and on those rank three arrangements supporting multinet structures with three classes \cite{FY07,Marco09}.   

For general arrangements we obtain bundles with nontrivial monodromy around each hyperplane by considering the product of bundles obtained from generating sets supported on members of a covering family \X\ of subarrangements of \A.  In the primary applications, \X\ is a covering family of rank two flats. For $S\in \X$,
with $\A_S \subset\A$ the subarrangement of hyperplanes containing $S$, 
one has an inclusion $M \subseteq M(\A_S)$. The induced homomorphism of fundamental groups has the effect of killing 
generators corresponding to the hyperplanes 
which do not contain 
$S$. To determine when the resulting representations may be faithful, we study the kernel of the homomorphism 
\[
\rho_\X \colon G \to \prod_{S \in \X} G_S,
\]
where $G$ is an arbitrary group with fixed finite generating set $Y$, $\X \subseteq 2^Y$ is a set of pairwise incomparable subsets of $Y$, and, for $S \in \X$, $G_S=G/\langle\langle Y-S\rangle\rangle$ is the quotient of $G$ by the normal subgroup $\langle\langle Y-S\rangle\rangle$ generated by $Y -S$. 

When $G=P_{\mu+1}$ is the pure braid group 
and $\rho_\X$ is the product of deletion maps $P_{\mu+1} \to  P_\mu$, one for each strand, the kernel of $\rho_\X$ is  
the group of Brunnian braids on $\mu+1$ strands -- see Example~\ref{ex:stanford}. In Theorem~\ref{thm:kernel} we extend an argument of T.~Stanford \cite{Stanf99} to the general setting described above, proving that the kernel of $\rho_\X$ is generated by iterated commutators of generators and their inverses, with the property that some entry or its inverse lies outside of $S$, for every $S \in \X$. The argument requires that the composite 
\[
\langle S \rangle \hookrightarrow G \to G_S
\]
is an isomorphism, for $S$ an intersection of elements of \X\ and for $S=\{y\}$, $y \in Y$. In this case \X\ is called a {\em retractive family} (relative to $G$ and $Y$). We establish some general sufficient conditions for retractive families, for application to arrangement groups -- in particular if $G_S$ is isomorphic to the free group $F_{|S|}$, or to $\Z \times F_{|S|-1}$, then $S$ is retractive. 

For $G=\pi_1(M)$, the set $Y$ is a {\em standard set of generators}, that is, a generating set of meridians, one for each hyperplane. The family $\X \subseteq 2^Y$ is identified with a subset of $2^\A$. To be retractive, it is necessary that \X\ be a set of flats of \A, corresponding to elements of the intersection lattice. We fix \X\ and vary the standard set of generators $Y$ so that \X\ corresponds to a retractive family relative to $G$ and $Y$. In this case we say $Y$ is {\em adapted to} $\X$. Families of modular flats are always retractive, and our result specializes to Stanford's description of Brunnian braids in terms of commutators -- see Example~\ref{ex:stanford}. A rank two flat $X$ is retractive if and only if the corresponding local commutator relations (the defining relations of $G_X$) are valid in the group $G$. This can be ensured for any single flat $X$ by choosing $Y$ appropriately; our results apply when \A\ is covered by a set of flats with this property, for some fixed standard set of generators $Y$. This is the case when \A\ has a ``conjugation-free" presentation, more precisely, when $G$ is a cyclically-presented hyperplane group \cite{Fried16} -- see Remark~\ref{rem:cycpres}.

Using the description of $\ker\left(\rho_\X \colon G \to \prod_{S \in \X} G_S\right)$ provided by Theorem \ref{thm:kernel}, 
we obtain in Theorem~\ref{thm:injtest}  a sufficient condition for $\rho_\X$ to be injective, for a retractive family \X, that involves only commutators of length two and three, and can often be verified by hand. In 
Section~3.4, 
we verify this condition for several well-studied arrangements, showing $\rho_\X$ is injective. As a consequence, for these examples, $G$ is isomorphic to a subgroup of a product of free groups, implying $G$ is residually free and has several other properties of interest -- see Theorem~\ref{thm:properties}.

If \A\ is a central arrangement, the complement $M$ is invariant under the diagonal $\C^\times$ action, and the quotient map $M \to M/\C^\times=:\bar{M}$ is a trivial $\C^\times$-bundle. The space $\bar{M}$ is the complement of the arrangement $\bar{\A}$ of projective hyperplanes $H/\C^\times$ in $\P^{\ell-1}(\C)$. One can choose a standard set of generators so that their product $z$ is central in $\pi_1(M)$. Then the {\em projective arrangement group} $\bar{G}=\pi_1(\bar{M})$ is isomorphic to $G/\langle z \rangle$.

Let \X\ be a set of pairwise incomparable flats of \A, and let $\Lambda_\X$ be the bipartite graph with vertex set $\A \cup \X$ and edges $\{H,X\}$ for $X \in \X$ and $H \in \A_X$. The homomorphism $\rho_\X$ induces a well-defined homomorphism $\bar{\rho}_\X \colon \bar{G} \to \prod_{S \in \X} \bar{G}_S$, with target a product of free groups. We show the image of $\bar{\rho}_\X$ is normal in $\prod_{S \in \X} \bar{G}_S$, provided $\Lambda_\X$ has no four-cycles, and we identify the cokernel of $\bar{\rho}$ with the first cohomology group $H^1(\Lambda_\X,\Z)$ of  $\Lambda_\X$. Since this group is free abelian, $\bar{\rho}(\bar{G})$ is the kernel of a homomorphism from a product of free groups to a free abelian group, that surjects onto each factor, that is, it is a subdirect product of free groups in the sense of \cite{BriMi09}. Then the results of \cite{MMV98} can be applied to determine the homological finiteness type of $\bar{\rho}(\bar{G})$, and of $\bar{G}$ when $\bar{\rho}$ is injective: if $\Lambda_\X$ is connected, then $\bar{\rho}(\bar{G})$ is of type $F_{m-1}$ but not $F_m$, where $m=|\X|$ -- see Theorem~\ref{type}. 

This result informs on a number arrangements of interest, including the 
rank three whirl arrangement, labelled $X_3$ in \cite{FaRa86}. Here, the group
$\bar{G}(\A)$ is isomorphic to the Stallings group of type $F_2$ but not of type $F_3$ \cite{Stall63}, as originally observed by Matei and Suciu \cite{MatSuc} -- see Example~\ref{ex:arvola3}. This observation provided motivation for the current project; 
see also Question 2.10 in Bestvina's problem list \cite{Best04}. 
The $X_3$ example was generalized by Artal, Cogolludo, and Matei \cite{ACM15} to a large family of arrangements whose groups are Bestvina-Brady groups. In Theorem~\ref{thm:bestbrad} we reproduce their result using our approach.

Failure of the retractive property for a particular rank two flat is reflected in the lower central series ranks of $G$. In the final section we exploit this phenomenon to establish a dichotomy for decomposable arrangements. This is a class of arrangements for which the lower central series factors are generated by local contributions. 
In Section~\ref{subsec:decomp} we prove that a decomposable arrangement group either has a standard generating set adapted to the set of all rank two flats, a type of ``conjugation-free presentation," or is not residually nilpotent.

\section{Discriminantal bundles} 

In this section we will construct a bundle $E_k(\A,\F) \to M(\A)$ over the complement of an arbitrary complex arrangement \A, depending on a positive integer $k$ and a family \F\ of functions on $M(\A)$ called a generating set. The total space can be identified with the complement of a family of extensions of \A, parametrized by the complement of \A\ via the generating set \F. 

The construction mimics the process of constructing the pure braid space for ${\ell + k}$ strings from the space for $\ell$ strings. The pure braid space $PB_\ell$ is the complement in $\C^\ell$ of the braid arrangement, defined by the polynomial $\prod_{1\leq i < j\leq \ell} (z_i-z_j) =0$, with fundamental group the pure braid group $P_\ell$.
Then $PB_{\ell +k}=\{(z_1,\ldots,z_{\ell},z_{\ell+1},\ldots,z_{\ell+k})\mid z_i \neq z_j\}$. 
To see how this arises from $PB_\ell$, let  $w_1\ldots, w_k$ be the last $k$ variables.
 Then 
\[
PB_{\ell+k}=
\{(z_1,\ldots,z_{\ell},w_{1},\ldots,w_{k})\mid z_i \neq z_j,z_i \neq w_j,w_i \neq w_j\}.
\]
We have introduced new variables, and prohibiting those variables from being equal or from taking values given by the ``generating set" $\F=\{z_1, \ldots, z_\ell\}$. The Fadell-Neuwirth bundle $PB_{\ell+k} \to PB_\ell$ is given by the projection $(z_1, \ldots, z_\ell, w_1, \ldots, w_k) \mapsto (z_1, \ldots, z_\ell)$, see \cite{FN}.  To generalize, we replace \F\ with an arbitrary collection of complex-valued continuous functions on $M$, whose pairwise differences are nowhere zero on $M$; these conditions ensure that the resulting projection is a bundle map.

\subsection{Construction of discriminantal bundles} 
Fix an arrangement of hyperplanes 
$\A=\{H_1,\ldots, H_n\}$  in $\C^\ell$ with  
$H_i$ the zero locus of the 
(possibly nonhomogeneous) 
linear polynomial $\alpha_i\colon \C^\ell \to \C$ 
for $1\leq i\leq n$. Let $Q=\prod_{i=1}^n \alpha_i$ and $M=M(\A)=\C^\ell - \bigcup_{i=1}^n H_i=\{(z_1\ldots, z_\ell) \in \C^\ell \mid Q(z_1,\ldots, z_\ell)\neq 0\}$.

\begin{definition}
A set $\F=\{f_1,f_2,\ldots,f_{\mu}\}$ is called a \emph{generating} set for the arrangement $\A$ provided that each $f_i$ is a complex-valued continuous
function on $M$ and each difference $f_i - f_j, \ 1\leq i<j\leq \ell$ is nowhere zero on $M$.
\end{definition}

So the coordinate functions $f_i(z)=z_i, \,1\leq i \leq \ell$, define a generating set for the braid arrangement in $\C^\ell$. Functions in a generating set need not be linear; in most examples the $f_i$ are rational functions on $\C^\ell$, regular on $M$. In this case \F\ forms a generating set for  \A\ if and only if the closures of irreducible components of the (possibly non-reduced) quasi-affine algebraic hypersurfaces defined by $f_i(z)=f_j(z)$ are hyperplanes of \A, for each $i\neq j$. For example, the functions $f_i(z)=z_i^{-2}$, $i=1,2$, form a generating set for the arrangement $\{z_1=0,z_2=0,z_1\pm z_2=0\}$.

Let $\F=\{f_1,\ldots,f_{\mu}\}$ be a generating set for $\A$ and let $k\geq 1$. For $z=(z_1, \ldots, z_\ell) \in M(\A)$ and $w=(w_1, \ldots, w_k) \in \C^k$, let 
\[
Q(z,w)=\prod_{\substack{1\leq i\leq k \\ 1\leq j\leq \mu}} (w_i - f_j(z)) \prod_{1 \leq i<j \leq k}(w_i-w_j),
\]
and define the topological space $E_k(\A,\F)$ by

\[
E_k(\A,\F)=\{(z,w) \in \C^\ell \times \C^k \mid z \in M(\A) \ \text{and} \ Q(z,w) \neq 0\}.
\] 

\begin{definition} The {\em discriminantal bundle} associated with \F\ and $k$ is the natural projection $p \colon E_k(\A,\F) \to M(\A)$.
\end{definition}

The terminology is justified by observations to follow. 
If the $f_i$ are meromorphic on $\C^\ell$, then the complement of $E_k(\A,\F)$ is an analytic hypersurface in $\C^{\ell+k}$, so $E_k(\A,\F)$ is a Stein manifold, hence has the homotopy type of a complex of dimension at most $\ell+k$.

By the discussion above, $PB_{\ell+k}=E_k(\A_\ell ,\F)$ where $\A_\ell$ is the braid arrangement in $\C^\ell$ and $f_i(z)=z_i$ for $1\leq i\leq \ell$. The projection $p \colon E_k(\A_\ell,\F) \to M(\A_\ell)$ is the Fadell-Neuwirth bundle $PB_{\ell+k} \to PB_\ell$.  In general we have the following result, generalizing \cite[Thm. 1.1.5]{Coh01}, which treats generating sets of  linear functions.

\begin{theorem}  \label{thm:pullback} 
If $\A$ is an arrangement with generating set $\F=\{f_1,\ldots,f_{\mu}\}$ and $k\geq 1$, 
then the 
projection $p\colon E_k(\A,\F) \rightarrow M(\A)$ is the projection map of a fiber bundle.  This bundle is the pullback of the Fadell-Neuwirth bundle $PB_{\mu +k} \rightarrow PB_{\mu}$ via the function $f = (f_1, \ldots , f_{\mu}) : M(\A) \rightarrow PB_{\mu}$.
\end{theorem}

\begin{proof} Let $z=(z_1, \ldots ,z_{\ell})$  and  $(x_1, \ldots ,x_\mu,w_{1}, \ldots , w_{k})$ be coordinates on $\C^\ell$ and $\C^{\mu+k}$,  respectively.  Then the total space of the pullback of $PB_{\mu +k} \rightarrow PB_{\mu}$ via $f$ is the set of all points 
\[
(z_1, \ldots , z_\ell, x_1, \ldots , x_\mu, w_1, \ldots , w_k)\in M(\A) \times PB_{\mu+k}
\]
which satisfy $f_i(z) = x_i, i =1, \ldots, \mu$.  It is readily checked that the map
\[
(z_1, z_2, \dots , z_{\ell}, w_1, \ldots , w_k) \longmapsto (z_1, \ldots , z_\ell, f_1(z), \dots , f_\mu(z),w_1, \ldots , w_k)
\]
from $E_k(\A,\F)$ to the total space of the pullback is a bundle equivalence. 
\end{proof}

The total space $E_k(\A,\F)$ can be interpreted as a space of labelled extensions $\A(z,w)$ of the arrangement \A\ by $k$ distinct hyperplanes, parametrized by complex numbers 
$w_i$, $1\le i\le k$, 
with forbidden values of the $w_i$ given by $f_1(z), \ldots, f_\mu(z)$ depending on $z \in M(\A)$. See Remark~\ref{extspace} for a concrete description in a particular example. 

The function $f = (f_1, \ldots , f_{\mu}) : M(\A) \rightarrow PB_{\mu}$, determined by the chosen labelling of \F, is called a \emph{generating function} for the discriminantal bundle; two generating functions for the same generating set \F\ differ by an self-homeomorpsism of $PB_\mu$, yielding isomorphic pullbacks. 

\begin{definition}  A \emph{discriminantal arrangement} of type $(\mu, k)$ is the arrangement $\A_{\mu,k}$ 
in $\C^k$, with coordinates $(w_1,\dots,w_k)$, 
defined by the polynomial \[\prod_{i,j} (w_i - m_j) \cdot \prod_{i<j}(w_i-w_j)\]
where $m_1, \ldots, m_\mu$ are fixed distinct complex numbers.
\end{definition}

The complement $F_{\mu,k}$ of $\A_{\mu,k}$ 
is the configuration space of $k$ ordered points in $\C-\{m_1,\dots,m_\mu\}$. Different choices of $m_1\ldots, m_\mu$ lead to lattice-isotopic arrangements. Thus the complement of $\A_{\mu,k}$ is determined up to homeomorphism by $\mu$ and $k$ \cite{R1}. We denote the complement of $\A_{\mu,k}$ by $F_{\mu,k}$. The arrangement $\A_{\mu,k}$ is an affine supersolvable arrangement, hence is itself a fiber-type arrangement. In particular $F_{\mu,k}$ is aspherical, see \cite{FR2,T2}.

Since $F_{\mu,k}$ is the fiber of the Fadell-Neuwirth bundle $PB_{\mu +k} \rightarrow PB_{\mu}$, the fundamental group $G_{\mu,k}=\pi_1(F_{\mu,k})$ is a normal subgroup of the pure braid group $P_{\mu+k}$.  Note that $G_{\mu,1}$ is a free group of rank $\mu$.

\begin{proposition}
The fiber of $p \colon E_k(\A,\F) \to M(\A)$ is homeomorphic to $F_{\mu,k}$.  
\end{proposition}
It is because of this observation that we call these discriminantal bundles.

\begin{corollary} 
The fiber of $p\colon E_k(\A,\F) \to M(\A)$ is an Eilenberg-Mac\,Lane space of type $K(G_{\mu,k},1)$.
\end{corollary}

\begin{remark} The bundle $p \colon E_k(\A,\F) \to M(\A)$ supports a fiber-preserving action of the symmetric group $S_k$, by permutation in the last $k$ variables, so the monodromy representations arising from this bundle are naturally $(\pi_1(M),S_k)$-bimodules.
\end{remark}

Write $E=E_k(\A,\F)$, $M=M(\A)$, and $G=G(\A)=\pi_1(M)$. The bundle map $p \colon E \rightarrow M$ is the restriction of a linear projection.   If the $f_i$ are linear and $k\geq 1$, then the total space is the complement of a hyperplane arrangement $\mathcal E$.  If, in addition, \A\ is central, then the map $p$ is the bundle projection associated with the modular flat \A\ of $\mathcal E$, see \cite{Par00,FP}. If $k=1$ then $p$ is a strictly linear fibration \cite{FR2,T2}, and $f$ is the associated root map as defined in \cite{CS2,Coh01}.
                                                                                                                             
Much of the topology of fiber-type arrangements carries over (but some does not, see Examples \ref{ex:2-plane} and \ref{ex:2-plane2}).  The results below all follow from the characterization of discriminantal bundles as pullbacks of Fadell-Neuwirth bundles, along with standard results for fiber bundles, see \cite{FR2,F4}. 

An arrangement \A\  is a {\em $K(\pi,1)$ arrangement} if $M(\A)$ is aspherical, and is a \emph{rational $K(\pi,1)$ arrangement} if the rational completion of $M(\A)$ is aspherical \cite{F4}. 

\begin{corollary} Let $p\colon E \to M$ be a discriminantal bundle. 
\begin{enumerate}
\item If \A\ is a $K(\pi,1)$ arrangement, then $E$ is apsherical. 
\item The bundle $p \colon E\rightarrow M$ has a section, and the action of $G$ on the fiber is trivial in homology.
\item The group $\pi_1(E)$ is isomorphic to a semidirect product $G_{\mu,k} \rtimes G$. 
\item The homology $H_*(E,\Z)$ is isomorphic to $H_*(M,\Z) \otimes H_*(F_{\mu,k},\Z)$.
\item If \A\ is a rational $K(\pi,1)$ arrangement, then the rational completion of $E$ is aspherical. 
\end{enumerate}
\end{corollary}

If $G=\pi_1(M(\A))$ is an iterated semidirect product of free groups, respectively, an almost-direct product of free groups \cite{FR2,CS5}, then so is $\pi_1(E)$.  In the latter instance, the cohomology ring of $\pi_1(E)$ may be calculated from the almost-direct product structure, see \cite{Coh10}.  Additionally, we have the following, as noted in \cite[Lem. 6.2]{CCP07}.

\begin{corollary}
If $\pi_1(M)$ is linear, then so is $\pi_1(E)$.
\end{corollary}
\begin{proof}
The group $\pi_1(E)$ is a subgroup of the product $\pi_1(M) \times \pi_1(PB_{\mu+k})$, which is linear since both factors are. 
\end{proof}

The fiber $F_{\mu,1}$ of the Fadell-Neuwirth bundle $PB_{\mu+1} \to PB_{\mu}$ is a copy of \C\  with $\mu$ punctures. The monodromy of the bundle is the (faithful) Artin representation $P_\mu=\pi_1(PB_\mu) \to \Aut(F_\mu)$ of the pure braid group in the group of automorphisms of the free group.  
With this identification, $P_\mu$ acts diagonally on $F_{\mu,k}$ for any $k\geq 1$, since the diagonal hyperplanes $w_i=w_j$ are preserved. The bundle $PB_{\mu+k} \to PB_\mu$ is the associated bundle of $PB_{\mu+1} \to PB_\mu$ with fiber $F_{\mu,k}$ with this action of $P_\mu$.
 
\begin{corollary}  The structure group of the bundle $E_k(\A,\F) \rightarrow M(\A)$ reduces to the pure braid group on $\mu=|\F|$ strings, and the bundle is associated with the discriminantal bundle $E_1(\A,\F) \to M(\A)$ corresponding to \F\ and $k=1$ via the diagonal action of $P_\mu$ on $F_{\mu,k}$.
\end{corollary}

\begin{definition} Let \F\ be a generating set for \A\ and $H \in \A$. We say \F\ is {\em trivial on $H$} if $f_i-f_j$ extends continuously and is not identically zero on $H$, for all $1\leq i<j \leq \mu$. The {\em support} of \F\ is the set of hyperplanes $H\in \A$ on which \F\ is not trivial.
\end{definition}

The monodromy of the bundle $E_k(\A,\F) \to M(\A)$ may be nontrivial around $H \in \A$  if $H$ is in the support of \F. If $S\subseteq \A$ denotes the support of \F, then \F\ is a generating set for $S$, and $E_k(\A,\F) \to M(\A)$ is a subbundle of $E_k(S,\F) \to M(S)$, the pullback by the inclusion map $M(\A) \hookrightarrow M(S)$.

\begin{example}  Let ${\mathcal D}_\ell$ be the Coxeter arrangement of type $D_\ell$, with defining equations $z_i\pm z_j=0$, $1 \leq i < j \leq \ell$.  Let $\F=\{f_1, \ldots,f_\ell\}$ where $f_i(z)= z_i^2$.  Then \F\ is a generating set on ${\mathcal D}_\ell$, with support ${\mathcal D}_\ell$, and $E_k({\mathcal D}_\ell,\F)$ is the complement of the union of $\ell(\ell-1)+\binom{k}{2}$ hyperplanes and $\ell$ affine quadrics in $\C^{\ell+k}$. In particular, $E_k({\mathcal D}_\ell,\F)$ is not a hyperplane arrangement complement.

The space $E_1({\mathcal D}_3,\F)$ is homeomorphic to an arrangement complement. The group $\pi_1(M({\mathcal D}_3))$ is isomorphic to the pure braid group $P_4$, and the homomorphism $f_*\colon P_4 \to P_3$ induced by $f$ is given on the standard pure braid generators by
\[
f_*(A_{i,j})= \begin{cases}
A_{1,2}&\text{if $\{i,j\}=\{1,2\}$ or $\{i,j\}=\{3,4\}$,}\\
A_{2,3}&\text{if $\{i,j\}=\{2,3\}$ or $\{i,j\}=\{1,4\}$,}\\
A_{1,3}&\text{if $\{i,j\}=\{1,3\}$,}\\
A_{1,2}A_{1,3}A_{1,2}^{-1}&\text{if $\{i,j\}=\{2,4\}$,}\\
\end{cases}
\]
with appropriate choices of basepoints.
This is the homomorphism on pure braid groups induced by the map $B_4 \to B_3$ on full braid groups sending $\sigma_1$ and $\sigma_3$ to $\sigma_1$, and $\sigma_2$ to $\sigma_2$. We showed in \cite{CFR11} that $f_*$ is equivalent to the homomorphism induced by the map $PB_4 \to PB_3$ defined by deleting the fourth strand. If follows that the total space $E=E_1({\mathcal D}_3,\F)$ is homeomorphic to the pull-back of the Fadell-Neuwirth bundle $PB_4 \to PB_3$ along itself; by a result of \cite{FP} this implies $E$ is diffeomorphic to the complement of an arrangement of nine hyperplanes in $\C^4$, the generalized parallel connection of ${\mathcal D}_3$ with itself along a modular line. (All such arrangements are isomorphic.) 

For general $k$ and $\ell$, we do not know whether $E_k({\mathcal D}_\ell,\F)$ has the homotopy type of an arrangement complement.
\label{ex:d3}
\end{example}

\begin{example}  
Let $\B_\ell$ be the Coxeter arrangement of type $B_\ell$, with defining equations $z_i=0$, $1\le i \le \ell$, and $z_i=\pm z_j, 1 \leq i < j \leq \ell$.  The set 
\[
\F=\{-z_\ell,\dots,-z_1,0,z_1,\dots,z_\ell\}
\]
is a generating set for $\B_\ell$, with support $\B_\ell$.  Here $0$ denotes the zero function, and $z_i$ a coordinate function as before. This linear generating set arises from the structure of the projection $M(\B_{\ell+1}) \to M(\B_\ell)$ as a strictly linear fibration. The corresponding generating function $f\colon M(\B_\ell) \to PB_{2\ell+1}$ realizes the bundle $M(\B_{\ell+1}) \to M(\B_\ell)$ as the pullback of the Fadell-Neuwirth bundle $PB_{2\ell+2}\to PB_{2\ell+1}$.  This is used to determine the structure of the type $B$ pure braid group $\pi_1(M(\B_\ell))$ as an almost-direct product of free groups in \cite[Thm. 1.4.3]{Coh01}.

The set $\F'=\{f_1\ldots, f_\ell\}$, where $f_i(z)=z_i^{-2}$ for $1\leq i\leq \ell$, is another generating set for the arrangement $\B_\ell$, with support the entire arrangement $\B_\ell$. Note that the hyperplanes $z_i=0$ and $z_j=0$ are poles of $f_i-f_j$, of multiplicity two.

Similar considerations apply to the reflection arrangements associated to the full monomial groups, with defining equations 
$z_i=0$, $1\le i \le \ell$, and $z_i=\zeta^k z_j$, $1 \leq i < j \leq \ell$, $1\le k\le r$, where $\zeta$ is a primitive $r^{\text{th}}$ root of unity.
\label{ex:recip}
\end{example}

\begin{example}[\cite{Bries73}] For the reflection arrangement of type $F_4$, with hyperplanes $z_i=0$, $1 \leq i \leq 4$, $z_i \pm z_j=0$, $1 \leq i<j \leq 4$, and $z_1 \pm z_2 \pm z_3 \pm z_4 =0$, the functions  
$f_i(z)=(z_1z_2z_3z_4)z_i^2$, $1 \leq i \leq 4$, comprise a generating set.
\end{example}

\begin{remark} Given a generating set \F\ as in the preceding examples, one can consider the monodromy of the bundle $p \colon E_2(\A,\F) \to M$, acting on the twisted homology group $H_2(F_{\mu,2}, R)$, where $R=\Q(s,t)$ is the $\pi_1(F_{\mu,2})$-module with the generators corresponding to hyperplanes $w_i-m_j=0$ acting by multiplication by $s$, and the generator corresponding to $w_1-w_2=0$ acting by multiplication by $t$.
The corresponding homology representation is an analogue of the Lawrence-Krammer-Bigelow  representation of the pure braid group for the arrangement group $G(\A)$; see \cite{Pa09}. We will call this the {\em LKB-type representation} of $G(\A)$ associated with $\F$. For the braid arrangement $\A_\ell$, with generating set $\F=\{z_1, \ldots, z_\ell\}$, this representation is indeed rationally equivalent to the Lawrence-Krammer-Bigelow representation of the pure braid group, by \cite{PP02}. For the examples above we do not know if the resulting representations are rationally equivalent to the LKB-type representations of the corresponding pure Artin groups as described in \cite{Pa09}. We did not examine whether Coxeter arrangements of types $E_n$, $n=6,7,8$, support generating sets.
\label{rem:lkb}
\end{remark}

\begin{example} \label{ex:2-plane}
Let \A\ be the arrangement consisting of the origin in $\C$. The set $\F=\{0,z,2\bar{z}\}$ is a generating set for \A. The total space of the discriminantal bundle $E_1(\A,\F)$ is the complement in $\C^2=\R^4$ of the real linear subspaces $z=0, w=0, w=z$, and $w=2\bar{z}$. This complement of four 2-planes in $\R^4$ does not have the homotopy type of the complement of a complex hyperplane arrangement, by \cite[Ex. 2.2]{Z5}. Furthermore, the group $G=\pi_1(E_1(\A,\F))$ is not the fundamental group of any smooth quasi-projective variety, see \cite[Ex. 8.1]{DPS08a}
\end{example}

\begin{example} \label{ex:2-plane2}
Again, let $\A=\{0\}$ in $\C$.
The set 
$\F=\{z,\frac{1}{2}z,
\frac{3}{8}z+\frac{5}{8}\bar{z}, -\frac{1}{12}z+\frac{5}{12}\bar{z}\}$
is a generating set for $\A$. 
Here, the space $E_1(\A,\F)$ is the complement in $\C^2=\R^4$ of the five 2-planes $z=0$, $w=z$, $2w=z$, $8w=3z+5\bar{z}$, and $12w=-z+5\bar{z}$.  
Noting 
that the equations $az+bw+c\bar{w}=0$ and $\alpha z+\beta\bar{z}+\gamma w=0$ define the same 2-plane if $\alpha(b^2-c^2)=\gamma a b$ and $\beta(b^2-c^2)=-\gamma a c$, this arrangement is identical to the arrangement considered in \cite[Ex. 8.2]{DPS08a}. There, it is shown that $G=\pi_1(E_1(\A,\F))$ is not a 1-formal group, and consequently that $E_1(\A,\F)$ is not a formal space.
\end{example}

The observation in the previous example may be used to establish the following.

\begin{proposition}
The complement of any 
arrangement of 2-planes through the origin
in $\C^2=\R^4$ may be realized as $E_1(\A,\F)$ for an appropriate choice of generating set $\F$ for the arrangement $\A=\{0\}$ in $\C$.
\end{proposition}

\subsection{Existence of generating sets}
\label{subsec:existence}
To construct faithful representations of the arrangement group $G(\A)=\pi_1(M)$, we need to know which subarrangements of \A\ support generating sets. It turns out the conditions are somewhat restrictive. But one can always construct generating sets supported on rank two subarrangements.

For an arrangement \A\ of affine hyperplanes in $\C^\ell$ and a subset $X$ of $\C^\ell$, let
\[
\A_X=\{H \in \A \mid H \supseteq X\}.
\]
The (finite) collection $\L(\A)=\{\A_X \mid X \subseteq \C^\ell\}\subseteq 2^\A$ forms a poset under inclusion; its elements are called {\em flats} of \A. The poset $\L(\A)$ has smallest element $\emptyset=\A_{\C^\ell}$ and largest element $\A=\A_\emptyset$. To avoid confusion we will call $\L(\A)$ the {\em poset of flats} of \A; it is an atomic lattice which is geometric if and only if $\bigcap_{H \in \A} H \neq \emptyset$, in which case we say \A\ is {\em central}. The poset of flats of a central arrangement is the lattice of flats of a simple matroid on \A. 

If $X \neq \emptyset$, the arrangement $\A_X$ is called the {\em localization} of \A\ at $X$. In any case we can assume without loss of generality that $X=\bigcap_{H \in \A_X} H$. Then $\A_X \subseteq \A_Y$ if and only if $X \supseteq Y$, and the mapping $\A_X \mapsto \bigcap_{H \in \A_X} H = X$ defines an isomorphism of $\L(\A)$ to the {\em (augmented) intersection poset}  $\LL(\A)$ of \A, the set of intersections of subsets of \A\ ordered by reverse inclusion. 
Observe that, for $\A$ non-central, the intersection poset $\LL(\A)$ has largest element $\emptyset$, contrary to the convention of  \cite{OT92}.
 The posets $\LL(\A)$ and $\L(\A)$ are ranked. A flat $\S\subseteq\A$ in $\L(\A)$ has rank $\rk(\S)$ equal to the codimension of $\bigcap_{H \in \S} H$. Similarly $X\in\LL(\A)$ has $\rk(X)=\codim(X)$. In particular, we have $\rk(X) = \rk(\A_X)=\codim(X)$ for $X \in \LL(\A)$ corresponding to $\A_X\in\L(\A)$.

The proof of the following result gives an indication of our original ideas for producing fibered families of hyperplanes, as discussed in Remark \ref{rem:r2} below. 

\begin{proposition} Let \A\ be an arbitrary arrangement and $X \in \LL(\A)$ with $\rk(X)=2$. Then there is a generating set for \A\ of size $|\A_X|$ supported on $\A_X$.
\label{prop:rank2gen}
\end{proposition}

\begin{proof} We may label the hyperplanes of \A\ so that $\A_X=\{H_1,\ldots, H_\mu\}, \, \mu=|\A_X|$. Since $H_1$ and $H_2$ are distinct we may choose coordinates so that they are given by $z_1=0$ and $z_2=0$, respectively. Then, since $X$ has codimension two, there are distinct, nonzero complex numbers $m_3, \ldots, m_\mu$, so that $H_i$ is defined by the equation $z_2=m_iz_1$. 
Define $\F_X=\{f_1, \ldots, f_\mu\}$, where 
\[
f_i(z) = \begin{cases}
0, & i=1 \\
z_2, & i=2 \\
m_iz_1,& 3 \leq i \leq \mu,
\end{cases}
\]
for $z \in M$, and observe that $\F_X$ is a generating set for \A\ with support equal to $\A_X$ and $|\F_X|=\mu$. 
\end{proof}

Abusing terminology slightly, we will call a generating set $\F_X$ constructed as in the proof of Proposition~\ref{prop:rank2gen} a {\em canonical} generating set for $X\in\LL(\A)$ with $\rk(X)=2$; it is determined by $X$ up to choice of coordinates.

\begin{remark}\label{rem:r2} For $X\in\LL(\A)$ with $\rk(X)=2$ and $\F_X$ a canonical generating set for $X$, a point $(z,w)$ in the space  $E_k(\A, \F_X)$  can be identified with the extension of \A\ by a pencil of $k$ distinct hyperplanes containing the codimension-two subspace $z+X$, with defining equations 
\[
x_2-z_2=w_i(x_1-z_1), 1 \leq i \leq k,
\]
 none of which are parallel to any of the hyperplanes in $\A_X$, with $z \in M$. Consequently, $E_k(\A,\F_X)$ is isomorphic to the space of all extensions of \A\ having a certain fixed intersection poset. Our work was motivated in part by the desire to understand the topology of ``nicely-behaved" extension spaces such as this.
\label{extspace}
\end{remark}

Generating sets supported on rank three flats are related to multinets (or combinatorial pencils) on projective line arrangements, as studied in \cite{Marco09,FY07}.

\begin{definition} A $(k,d)$-{\em multinet} on a rank three central arrangement \A\ consists of a function $m\colon \A \to \Z_{>0}$ and a  partition $\PPP$ of \A\ with $k\geq 3$ blocks, with the associated {\em base locus} $\X$ being the set of rank two flats not contained in parts of \PPP, satisfying

\begin{enumerate}
\item for each $\pi \in \PPP$, $\sum_{H \in \pi} m(H)=d$;
\item for each $x \in \X$, $\sum_{H \in \pi, x \in H} m(H)$ is independent of $\pi\in \PPP$;
\item $(\bigcup_{H\in \pi} H) - \X$ is connected for each $\pi \in \PPP$.
\end{enumerate}
\end{definition}

By results of \cite{Marco09,FY07}, \A\ supports a $(k,d)$-multinet structure if and only if there is a linear system of curves of degree $d$ on $\CP^2$ with no fixed components having $k$ completely reducible fibers whose irreducible components are the projectivizations of the hyperplanes in \A. By results of \cite{PerYuz07}, such pencils exist only for $k \leq 4$.

\begin{theorem} Let \A\ be a rank three arrangement supporting a $(3,d)$-multinet for some $d \geq 1$. Then there is a generating set $\{f_1,f_2,f_3\}$ with support equal to \A. 
\end{theorem}

\begin{proof} By \cite{FY07}, there are pairwise relatively prime, homogeneous, completely reducible polynomials $Q_i$, $1 \leq i \leq 3$, of degree $d$, satisfying 
\[
\{z \in \C^\ell \mid Q_1(z)Q_2(z)Q_3(z)=0\}=\bigcup_{H \in \A} H, 
\]
and $Q_3=Q_1-Q_2$.
Then $\{0,Q_1,Q_2\}$ is a generating set with support equal to \A. 
\end{proof}

There is a partial converse. Suppose $\{f_1,f_2, f_3\}$ is a generating set consisting of homogeneous polynomials of the same degree $d$, with support $\A$. If the linear system generated by $f_1-f_2$ and $f_2-f_3$ has no fixed components and connected general fiber, then \A\ supports a $(3,d)$-multinet structure.

Any set $\F=\{f_1,\ldots, f_\mu\}$ of $\mu$ distinct linear forms is a generating set whose support \A\ has defining polynomial $\prod_{1\leq i<j\leq \mu}(f_i-f_j)$. In this case, as observed earlier, $E_k(\A,\F)$ is the complement of an arrangement $\mathcal E$, which contains \A\ as a modular flat, and $E_k(\A,\F) \to M(\A)$ is the associated bundle projection. 

\begin{example}
If \F\ consists of the coordinate functions $\{z_1,\ldots, z_\ell\}$ then its support  \A\ is the braid arrangement. If \F\ consists of the natural defining forms $z_i-z_j$ for the braid arrangement, then \A\ is the $p=2$ center-of-mass arrangement defined in \cite{CohKam07}, whose complement parametrizes the  labelled configurations of $\ell$ distinct points in $\R^2$ with pairwise distinct midpoints. More generally, if \F\ consists of the natural defining forms $\sum_{k=1}^p z_{i_k} - \sum_{k=1}^p z_{j_k}$, for the $p$-fold center-of-mass arrangement on $\ell$ points, then the supporting arrangement \A\ is the $2p$-fold center-of mass arrangement on $\ell$ points.
\end{example}

If $f_1, \ldots, f_\mu$ are distinct linear forms, then the set $\F=\{\frac{1}{f_1}, \ldots, \frac{1}{f_\mu}\}$ of reciprocals is also a generating set, whose support is the arrangement defined by $$\prod_{1\leq i \leq \mu} f_i \cdot \prod_{1\leq i<j\leq \mu}(f_i-f_j).$$ 

\section{Products of localization homomorphisms}
\label{sec:prods}

Given an arbitrary arrangement \A\ of rank at least two, we would like to build a bundle with base $M(\A)$ which is sufficiently twisted to yield a faithful representation of $\pi_1(M)$. For \A\ itself to support a discriminantal bundle requires fairly special circumstances, as we have seen, but \A\ may have several proper subarrangements supporting such bundles. Indeed, any rank two subarrangement, and any rank three subarrangement supporting a multinet, will have that property. We propose to pull back the product of all such discriminantal bundles supported on subarrangements, to obtain a bundle over $M(\A)$. 

If \D\  denotes the set of subarrangements of \A\ supporting generating sets, let
$$\phi_\D  \colon M(\A)  \to \prod_{S\in \D} M_S$$ be the product of inclusion maps, where $M_S$ is the complement of the hyperplanes in $S\subset \A$. Note that the codomain is also an arrangement complement. Choosing a generating set $\F_S$ of size $\mu_S$ and a positive integer $k_S$ for each $S\in \D$, we have discriminantal bundles $E_{k_S}(S,\F_S) \to M_S$, and hence a product bundle $$\prod p_S \colon \prod_{S\in \D} E_{k_S}(S,\F_S) \to \prod_{S \in \D} M_S.$$ The pullback $\phi_\D^*(\prod p_S)$ gives a bundle over $M(\A)$, with fiber $F=\prod_{S\in \D} F_{k_S,\mu_S}$, and nontrivial monodromy around each hyperplane in $\bigcup _{S \in \D} S$.

To use the product bundle $\phi_\D^*(\prod p_S)$ constructed above to produce faithful representations of $\pi_1(M(\A))$, one might first build faithful representations of $\pi_1(M_S)$ for $S\in \D$, using the monodromy of discriminantal bundles, and then attempt to show that $\phi_\D$ induces an injection on fundamental groups. We can carry out the first step at least in case $S$ comes from a rank two flat.

\subsection{Representations associated to rank two flats} \label{sec:r2}
Fix now a single element $X\in\LL(\A)$ of rank two.  Let $\A_X=\{H \in \A \mid H \supseteq X\}$, and let $M_X$ be the complement of $\A_X$. Let  $i_{X}: \pi_{1}(M) \rightarrow \pi_{1}(M_X)$ be the homomorphism induced by the inclusion $M \hookrightarrow M_X$. 
Let $\F_X=\{f_1,\ldots,f_\mu\}$, $\mu=|\A_X|$, be a canonical generating set associated with $X$, as constructed in the proof of Proposition~\ref{prop:rank2gen}, and denote the associated generating function by $f_X \colon M_X  \rightarrow PB_{\mu}$. 

\begin{proposition}  The induced homomorphism $(f_X)_* \colon \pi_1(M_X) \to P_\mu$ is injective.
\end{proposition}
\begin{proof}
Assume without loss of generality that $\A_X=\{H_1,\ldots, H_\mu\}$, with $H_1$ is defined by $z_1=0$, $H_2$ by $z_2=0$, and $H_i$ is defined by $z_2=m_iz_1$ for $3\leq i\leq \mu$. Then  $f_X\colon M_X  \rightarrow PB_{\mu}$ is given by $f_X(z)=(0,z_2,m_3z_1, \ldots, m_\mu z_1)$. Let $p\colon PB_\mu \to PB_{\mu-1}$ be the projection $(u_1,\ldots, u_\mu) \mapsto (u_1,u_3, \ldots, u_\mu)$. Then $p$ is a discriminantal bundle projection; in particular the kernel of $p_*$ is the fundamental group of the fiber of $p$, a free group on $\mu-1$ generators.

For $i=1,\ldots, \mu$, let $y_i$ be a loop in $M_X$ about $H_i$.  Choosing the base point in the hyperplane $z_1=1$, we may assume that the loops $y_2, \ldots, y_\mu$ lie in the subspace $z_1=1, z_3=m_3, \ldots, z_\mu=m_\mu$. Then $p_*\circ (f_X)_*$ sends $y_i$ to $1$ for $2 \leq i \leq \mu$. Then any element of the kernel of $(f_X)_*$ lies in $\langle y_2 \ldots, y_\mu\rangle$, a free group of rank $\mu-1$. Moreover, $(f_X)_*$ sends this subgroup to the fundamental group of the fiber of $p$. Then $(f_X)_*$ is injective by the Hopfian property of free groups. 
\end{proof}

Recall the LKB-type representation of $G(\A)$ over $\Q(s,t)$ associated with a generating set \F, defined in Remark~\ref{rem:lkb}.
 
\begin{corollary} The LKB-type representation of $G(\A_X)$ associated with a canonical generating set $\F_X$ for $X$ is faithful.
\label{thm:rank2faith}
\end{corollary}

\begin{proof} Since $p$ is the pullback of the Fadell-Neuwirth bundle $PB_{\mu+2} \to PB_\mu$ via $f_X$, and $(f_X)_*$ is injective, 
this follows from the faithfulness of the LKB representation of the braid group \cite{Big01,Kra02,PP02}.
\end{proof}

\subsection{The kernel of \texorpdfstring{$(\phi_\sS)_*$}{}}
\label{subsec:Stan}
Next we consider a general product mapping 
\[
\phi_\sS = \prod_{S\in \sS} i_S \colon M \to \prod_{S \in \sS} M_S,
\]
where \sS\ is an arbitrary set of subarrangements of \A\ and $i_S \colon M \to M_S$ is the inclusion. We pick a base point in $M$ and obtain an induced homomorphism, \[ (\phi_\sS)_* \colon \pi_{1}(M) \rightarrow \prod_{S\in\sS}{\pi_1(M_S)}. \] 
For simplicity, denote $(\phi_\sS)_*$ by $\rho_\sS$ and $(i_S)_*$ by $\rho_S$, so that $\rho_\sS=\prod_{S\in \sS} \rho_S$.  We adopt the conventions $x^y=y^{-1}xy$ and $[x,y]=x^{-1}y^{-1}xy$ for group elements $x$ and $y$, in agreement with \cite{MKS76} and \cite{F3}.

Recall that $\pi_1(M)$ is generated by small loops around the hyperplanes of \A. For each $S\in \sS$, $\rho_S$ kills the generators corresponding to hyperplanes in $\A - S$. For the pure braid group, and a certain sets of flats \sS, this has the effect of deleting strands. So elements in the kernel of the product mapping $\rho_\sS$ are analogous to Brunnian braids, braids that become trivial upon deletion of any strand. 

\begin{example}
Let \A\ be the braid arrangement in $\C^4$, so that $\pi_1(M) = P_4$, the four strand pure braid group.  Denote the pure braid generators by $A_{ij}$, for $1\leq i<j \leq 4$, corresponding to the hyperplanes $H_{ij}$ given by $x_i=x_j$. By considering the projection to $\C^3$ along the $x_4$ axis, we see that the subgroup $U$ generated by $A_{14}, A_{24}$ and $A_{34}$ is a free subgroup on three generators.  

Let $\sS=\{S_{123},S_{124},S_{134},S_{234}\}$ be the set of rank two flats of \A\ of multiplicity greater than two: $S_{ijk}=\{H_{ij},H_{ik}, H_{jk}\}$. Consider the commutator $g=[A_{14},[A_{24},A_{34}]]$.  Then, for every $i,j,k$, $\rho_{S_{ijk}}(g)=1$. Indeed, one of $H_{14}, H_{24}$ or $H_{34}$ lies outside of $S_{ijk}$, hence at least one factor of the commutator is sent to 1 by $\rho_{S_{ijk}}$. Thus $\rho_\sS(g)=1$. Clearly $g \neq 1$, since $g$ is a nontrivial reduced word lying in the free subgroup $U$. Consequently, $\rho_\sS$ is not injective in general. Interpreted as a map on pure braids, the homomorphism $\rho_\sS$ has the effect of deleting, in turn, each of the four strands, Thus $g$ corresponds to a nontrivial Brunnian pure braid on four strands. (The closure of this braid is the Borromean rings link.)
\label{ex:brunnian}
\end{example}

\begin{remark} The argument used in this example generalizes easily to show that $\rho_\sS$ is not injective for the complement of any strictly linearly-fibered arrangement which is not a product, whose fiber is the complement of $k\geq 3$ points in \C, and any set of flats \sS.
\end{remark}

Stanford showed that any braid which becomes trivial upon deletion of the strands outside a set $S \subseteq \{1,\ldots, \ell\}$ is a product of iterated commutators of pure braid generators $A_{ij}$ and their inverses 
at least one of which satisfies $\{i,j\} \not \subseteq S$ -- see Example~\ref{ex:stanford}. Stanford's argument can be cast in a more general setting so as to apply to other groups, including arrangement groups.

We formulate Stanford's notion of monic commutator \cite[Def. 1.3]{Stanf99} in terms of the bracket arrangements of \cite{MKS76}. A {\em bracket arrangement of weight $p$}  is a sequence $\beta^p$ of brackets and asterisks defined recursively as follows:

\begin{enumerate}
\item $\beta^1= [*]$ is the unique bracket arrangement of weight one.
\item For $p>1$, $\beta^p=[\beta^k\beta^\ell]$ for some bracket arrangements $\beta^k$ and $\beta^\ell$ of weights $k\geq 1$ and $\ell \geq 1$, respectively, satisfying $k+\ell=p$. 
\end{enumerate}

For instance, $[*[**]]$ and $[[**]*]$ are the two bracket arrangements of weight 3, where $[*]$ has been abbreviated to $*$. Given a group $G$, an ordered $p$-tuple $(x_1, \ldots, x_p)$ of elements of $G$, and a weight $p$ bracket arrangement $\beta^p$, the element $\beta^p(x_1, \ldots, x_p)$  of $G$ is defined recursively by $\beta^1(x_1)=x_1$, and, for $p>1$ and $\beta^p=[\beta^k\beta^\ell]$ with $p=k+\ell$, $\beta^p(x_1, \ldots, x_p)=[\beta^k(x_1 \ldots, x_k),\beta^\ell(x_{k+1}, \ldots, x_p)]$. 

\begin{definition} Let $G$ be a group with set of generators $Y$. A {\em monic commutator} in $G$, relative to $Y$, is a nontrivial element $q\in G$ of the form $\beta_p(y_1^{\epsilon_1}, \ldots, y_p^{\epsilon_p})$, where $p \geq 1$, $\beta^p$ is a bracket arrangement of weight $p$, $y_i\in Y$ and $\epsilon_i=\pm 1$ for $1 \leq i \leq p$. 
\end{definition}
We omit the phrase ``relative to $Y$," when the set of generators $Y$ has been fixed. There is no assumption that the $y_i$ are pairwise distinct.

Let $F=F(Y)$ be the free group on the set $Y$. For $w \in F$ the support $\supp(w)$ of $w$ is the set of generators appearing in the unique reduced word in $Y^{\pm 1}$ representing $w$.

\begin{proposition} \begin{enumerate}
\item For any $p\geq 1$, any ordered $p$-tuple $(w_1, \ldots, w_p)$ of elements of $F$, and bracket arrangement $\beta^p$, the element $\beta^p(w_1, \ldots, w_p)$ of $F$ is in the kernel of the projection $F \to F/\langle\langle w_i \rangle\rangle$, for all $1 \leq i \leq p$. 
\item If $w=\beta^p(y_1^{\epsilon_1}, \ldots, y_p^{\epsilon_p})$  is a monic commutator in $F$, then $\supp(w)=\{y_1, \ldots, y_p\}$.
\end{enumerate}
\label{poof}
\end{proposition}

\begin{proof} (i) The statement is true if $p=1$. If $p>1$ and $\beta^p=[\beta^k\beta^\ell]$ with $1 \leq k$, $1 \leq \ell$, and $k+\ell=p$, then by induction $\beta^k(w_1, \ldots, w_k)$ or $\beta^\ell(w_{k+1}, \ldots, w_p)$ vanishes modulo $w_i$, since $1 \leq i \leq k$ or $k<i\leq p$. Then the commutator 
\[
[\beta^k(w_1 \ldots, w_k),\beta^\ell(w_{k+1}, \ldots, w_p)]=\beta^p(w_1, \ldots, w_p)
\] 
also vanishes modulo $w_i$.\\
(ii) Since $F \to F/\langle\langle y_i^{\epsilon_i} \rangle\rangle$ 
restricts to an injection on $\langle y_j \mid y_j \neq y_i \rangle$ 
for each $i$ and $w \neq 1$ by hypothesis, we must have $y_i \in \supp(w)$ for $1 \leq i \leq p$. 
\end{proof}

\begin{definition} A subset \sS\ of $2^Y$ is {\em meet-closed} if it is closed under intersection, and is {\em atomic} if $\{y\} \in \sS$ for every $y \in Y$.
\end{definition}

Henceforth assume $2 \leq |Y|<\infty$. An atomic meet-closed family $\sS \subseteq 2^Y$ forms an atomic meet-semilattice under the inclusion partial order, with smallest element $\emptyset$. By convention, the intersection $Y$ of the empty family is an element of \sS. Then, since $Y$ is finite, joins exist in \sS\ as well, so \sS\ is an atomic lattice under inclusion.

\begin{definition} Let $\sS \subseteq 2^Y$ be an atomic lattice, and $w \in F(Y)$. The {\em support relative to \sS of $w$} is 
\[
\supp_\sS(w) =\bigcap\{S \in \sS \mid \supp(w) \subseteq S\}.
\]
\end{definition}

Again we will omit the phrase ``relative to \sS" when there is no risk of confusion.

\begin{lemma} \label{lem:monlem} 
Let $\sS=\{S_1, \ldots, S_r\} \subseteq 2^Y$ be an atomic lattice, linearly-ordered so that $S_i \subseteq S_j$ implies $i \leq j$. Then every $w \in F(Y)$ has a factorization  $w=w_1\cdots w_r$, where $w_i$ is either 1 or a product of monic commutators with support relative to $\sS$ equal to $S_i$ for $1 \leq i \leq r$.
\end{lemma}

\begin{proof} 
Observe that $S_1=\emptyset$ and $S_r=Y$.

We prove by induction on $k$ that $w$ has a factorization $w=w_1 \ldots w_ks$ with $w_i$ equal to 1 or a product of monic commutators with support equal to $S_i$, for $1 \leq i \leq k$, and $s$ equal to 1 or a product of monic commutators with support strictly greater than $S_k$ in the linear order. Since $\sS$ is atomic and elements of $Y^{\pm 1}$ are monic commutators by definition, there is such a factorization $w=w_1s$ with $w_1=1$ and $s=w$. 

Assume inductively that $k \geq 1$ and $w=w_1 \cdots w_ks$ as above. The case $s=1$ is trivial, so assume $s\neq 1$ and fix a factorization of $s$ into monic commutators as described. We show $s=uv$, where $u$ is 1 or a product of monic commutators with support $S_{k+1}$ and $v$ is 1 or a product of monic commutators with support greater than $S_{k+1}$, by induction on the number of monic commutator factors in $s$ with support equal to $S_{k+1}$. If no monic commutator factor in $s$ has support equal to $S_{k+1}$, the statement holds with $u=1$ and $v=s$. 

Suppose that some monic commutator factor occurring in $s$ has support equal to $S_{k+1}$. Then $s=s_1\cdots s_tyz$ where each $s_i$ is a monic commutator with support strictly greater than $S_{k+1}$, $y$ is a monic commutator with support $S_{k+1}$, and $z$ is a product of monic commutators with support greater than or equal to $S_{k+1}$. Then, following Stanford \cite{Stanf99}, $$s=ys_1[s_1,y]s_2[s_2,y]s_3 \cdots s_t[s_t,y]z.$$ Then $s=ys'$ with $s'=s_1[s_1,y]s_2[s_2,y]s_3 \cdots s_t[s_t,y]z$. Using Proposition~\ref{poof} (ii), we see that $s'$ has one fewer monic commutator factors with support $S_{k+1}$ than $s$. By the inductive hypothesis, $s'=u'v'$ with $u'$ equal to 1 or a product of monic commutators with support $S_{k+1}$ and $v'$ equal to 1 or a product of monic commutators with support strictly greater than $S_{k+1}$. Setting $u=yu'$ and $v=v'$, we have the 
desired factorization $s=uv$.

Then, setting $w_{k+1}=u$ and $s'=v$, we have $w=w_1 \cdots w_{k+1}s'$, with $w_i$ equal to 1 or a product of monic commutators with support $S_i$, for $1 \leq i \leq k+1$, and $s'$ equal to 1 or a product of monic commutators with support strictly greater than $S_{k+1}$. This completes the initial inductive argument; the statement of the lemma is obtained by setting $k=r$. 
\end{proof}

The image of $w \in F$ under the surjection $F \to G$ is denoted $\bar{w}$. Any monic commutator in $G$ is the image of  a monic commutator in $F$. Let $\sS\subseteq 2^Y$ be an atomic lattice as above.

\begin{definition} 
The monic commutator $q$ in $G$ is {\em transverse to} $\sS$ if $q=\bar{w}$ for some monic commutator $w$ in $F$ with $\supp_\sS(w)=Y$.
\label{bigfam}
\end{definition}

For $S\subseteq Y$, let $G_S=G/\langle\langle Y- 
S \rangle\rangle$. 
The canonical projection $\rho_S: G \to G_S$ then 
trivializes
generators of $G$ not in $S$. Note that $\rho_Y \colon G \to G_Y$ is the identity map, and $\rho_\emptyset$ is trivial.

For an arbitrary family $\X \subseteq 2^Y$, let 
\[
\rho_\X  = {\textstyle \underset{S\in \X}{\prod}} \rho_S: G \to \prod_{S\in \X} G_S.
\]
Given such a family $\X$, let $\sS=\bar{\X}$ denote the intersection of all atomic lattices in $2^Y$ containing \X. Then \sS\ consists of the elements of \X\ and all their finite intersections, along with the singletons $\{y\}$ for $y \in Y$, $\emptyset$, and $Y$. Conversely, given an atomic lattice $\sS \subseteq 2^Y$, the set $\X=\max(\sS)$ of elements of \sS\ covered by $Y$ is the minimal subset of \sS\ satisfying $\bar{\X}=\sS$. Note $\X=\max(\sS)$ is an anti-chain (or clutter): its elements are pairwise incomparable.

\begin{proposition} Suppose $\sS \subseteq 2^Y$ is an atomic lattice and $\X = \max(\sS)$. Then $\rho_\X$ and $\rho_{\sS-\{Y\}}$ have the same kernel.
\label{bigsmall}
\end{proposition}

\begin{proof} Since $\X \subseteq \sS-\{Y\}$, $\ker(\rho_{\sS-\{Y\}})\subseteq \ker(\rho_\X)$. Let $T \in \sS-\{Y\}$. Then there exists $S \in \X$ with $S \supseteq T$. One has a natural projection $\rho_{S,T} \colon G_S \to G_T$, in which the elements of $S - T$ are mapped to 1.  Then $\rho_{\sS-\{Y\}}$ can be factored as $\rho_{\sS-\{Y\}}=\psi \circ \rho_\X$, where $\psi \colon \prod_{S \in \X} G_S \to \prod_{T \in \sS-\{Y\}} G_T$ has components $\rho_{S,T}$ if $S \supseteq T$, and $1$ otherwise. Then $\ker(\rho_\X) \subseteq \ker(\rho_{\sS-\{Y\}})$. 
\end{proof}

\begin{definition} Let $\X \subseteq 2^Y$. A monic commutator $q$ is {\em transverse to \X} if there exists $w \in F$ with $q=\bar{w}$ and $\supp(w) \not \subseteq S$ for every $S \in \X$.
\end{definition}

Then $q$ is transverse to \X\ if and only if $q$ is transverse to the atomic lattice $\sS=\bar{\X}$ according to Definition~\ref{bigfam}.

\begin{proposition} If $q \in G$ is a monic commutator transverse to \X, then $q$ is in the kernel of the homomorphism $\rho_\X \colon G \to \prod_{S \in \X} G_S$.
\label{oneway}
\end{proposition}

\begin{proof} 
By hypothesis, $q=\bar{w}$ for some monic commutator $w=\beta^n(y_1^{\epsilon_1}, \ldots, y_p^{\epsilon_p})$ 
such that, for every $S \in \X$, there is an $i$ such that $y_i \not \in S$. The composite $F \to G \overset{\rho_S}{\longrightarrow} G_S$ factors through $F \to F/\langle\langle y_i^{\epsilon_i} \rangle\rangle$. Then $\rho_S(q)=1$ by Proposition~\ref{poof} (i), and consequently $\rho_\X(q)=1$.
\end{proof}

\begin{definition} A subset $S$ of $Y$ is {\em retractive} if $\rho_S \colon G \to G_S$ restricts to an injection $\langle S \rangle \to G_S$.  
\end{definition}
If $S$ is retractive then $\rho_S|_{\langle S\rangle}\colon \langle S \rangle \to G_S$ is an isomorphism.

A subset $S$ of $Y$ is retractive if and only if relations satisfied by $S$ in $G_S$, where generators outside of $S$ have been set equal to 1, are satisfied in $G$. If $S$ is retractive, we may identify $\langle S \rangle$ with $G_S$.  Then $\rho_S \colon G \to G_S$ is a retraction in the usual sense: $\rho_S(q)=q$ for $q\in G_S$.  The set $Y$ itself is always retractive.

\begin{theorem} Let $\sS \subseteq 2^Y$ be an atomic lattice consisting of retractive sets, and let $\X=\max(\sS)$. Then the kernel of $\rho_\X$ is generated by the monic commutators in $G$ transverse to \X.
\label{thm:kernel}
\end{theorem}

\begin{proof} By Proposition~\ref{oneway}, every such monic commutator lies in the kernel of $\rho_\X$. For the converse, suppose $q \in \ker(\rho_\X)$. As in Lemma \ref{lem:monlem}, write $\sS=\{S_1 \ldots, S_r\}$ with $i\leq j$ when $S_i \subseteq S_j$, and write $q=q_1 \ldots q_r$, where $q_i=\bar{w_i}$ with $w_i \in F$ equal to 1 or a product of monic commutators with support equal to $S_i$, for $1 \leq i \leq r$. We prove $q_j=1$ for $1 \leq j <r$ by induction on $j$, with the result that $q=q_r=\bar{w}_r$, and $\supp_\sS(w_r)=S_r=Y$, so that $q$ is a product of monic commutators transverse to \sS. We have $q_1=1$. Suppose $j \geq 2$ and $q_i=1$ for $1 \leq i<j$. Since $\rho_\X(q)=1$, we have $\rho_{S_j}(q)=1$ by Proposition~\ref{bigsmall}. If $1 \leq i<j$, then $\rho_{S_j}(q_i)=\rho_{S_j}(1)=1$. For $j<k<r$, we have $q_k=\bar{w}_k$ and $\supp_\sS(w_k)=S_k$. By the assumption on the linear order, $S_k \not \subseteq S_j$. Then $\supp(w_k) \not \subseteq S_j$, and hence $\rho_{S_j}(q_k)=1$ by Lemma~\ref{poof}.

Then $1=\rho_{S_j}(q)=\rho_{S_j}(q_j)$. Since $S_j$ is retractive, this implies $q_j=1$. This completes the inductive step. 
\end{proof}

The hypothesis that \sS\ consist of retractive sets is necessary. For example, 
let $G=\langle a,b \mid aba=bab\rangle$ be the three strand full braid group, with $S= \{a\}$. 
Then $G_S$ is trivial, and  $a\neq 1$ is in $\ker(\rho_S)$, but is not represented by a product of monic commutators whose supports contain $b$, as can be seen by considering induced permutations.

\begin{example}
Let $Y=\{A_{ij} \mid 1 \leq i< j \leq \ell\}$ be the standard set of generators for $G=PB_\ell$, the pure braid group on $\ell$ strands. For $I \subseteq \{1, \ldots, \ell\}$, let $S_I=\{A_{ij} \mid i, j \in I\}$. The homomorphism $G \to G_{S_I}$ corresponds to the map on braids defined by deleting the strands whose labels are not in $I$. That $S_I$ is retractive can be seen by pulling the strands with labels not in $I$ off to the side -- see also  Theorem~\ref{thm:localret}.

In \cite{Stanf99}, it is shown that the group of Brunnian braids, those which become trivial upon deletion of any strand, is generated by monic commutators whose entries involve every strand. This result now follows from Theorem~\ref{thm:kernel}, 
by taking 
$\X=\{S_I \mid |I|=\ell-1\}$ and $\sS=\left\{S_I \mid I \subseteq \{1, \ldots, \ell\}\right\}$. 
\label{ex:stanford}
\end{example}

Finally, we establish a condition for the injectivity of $\rho_\X$ that one can often check by hand. 

\begin{theorem} Let $\sS \subseteq 2^Y$ be an atomic lattice consisting of retractive sets, and let $\X=\max(\sS)$. 
Then $\rho_\X$ is injective if
\begin{enumerate}
\item $[a,b]=1$ for all $a,b \in Y$, $\{a,b\} \not \subseteq S$ for all $S \in \X$, and
\item $[a,[G_S,G_S]]=1$ for every $S\in \X$ and  $a\in Y-S$.
\end{enumerate}
\label{thm:injtest}
\end{theorem}

\begin{proof} Assume (i) and (ii) are satisfied. 
Let $w=\beta^p(y_1^{\pm 1}, \ldots, y_p^{\pm 1})$ be a monic commutator of weight $p$ in $F$, transverse to \X. We show $\bar{w}=1$ in $G$ by induction on $p$. Since \sS\ includes the singletons and $\supp(w)=\{y_1, \ldots, y_p\}$ is transverse to \sS, $p \geq 2$. The case $p=2$ is covered by the first hypothesis. Suppose $p>2$. Then $w=[u,v]$ for monic commutators $u$ and $v$ in $F$ of weights $k \geq 1$ and $\ell \geq 1$, with $k+\ell=p$. We may assume without loss that $\ell \geq 2$. If $v$ is transverse to \X\ then $\bar{v}=1$ by the inductive hypothesis, and then $\bar{w}=1$. Thus we may assume $\supp(v)\subseteq S$ for some $S\in \X$. Since $\ell \geq 2$, $\bar{v} \in [G_S,G_S]$. 

Write $u=u_0s_1 u_1 s_2 \cdots s_ku_k$ with $\supp(u_i)\subseteq Y-S$ for $0 \leq i \leq k$ and $\supp(s_i) \subseteq S$ for $1 \leq i \leq k$.  Repeated application of (ii) shows $[\bar{u}_i,\bar{v}]=1$ for $0 \leq i \leq k$. Writing $x \sim y$ in place of $\bar{x}=\bar{y}$, we show 
\[
[u_0s_1 u_1 s_2 \cdots s_iu_i,v]\sim [s_1s_2 \cdots s_i,v]
\]
by induction on $i$. The statement holds for $i=0$ by the observation above. For the inductive step, 
if $[u_0s_1 \cdots s_{i-1}u_{i-1},v]\sim [s_1 \cdots s_{i-1},v]$, we have
\begin{align*}
[u_0s_1 \cdots s_iu_i,v] &= u_i^{-1}[u_0s_1 \cdots s_{i-1}u_{i-1}s_i,v]u_i[u_i,v]
\sim  u_i^{-1}[u_0s_1 \cdots s_{i-1}u_{i-1}s_i,v]u_i\\
& = u_i^{-1}s_i^{-1}[u_0s_1 \cdots s_{i-1}u_{i-1},v]v^{-1}s_ivu_i
\sim u_i^{-1}s_i^{-1}[s_1 \cdots s_{i-1},v]v^{-1}s_ivu_i\\
&= u_i^{-1}[s_1 \cdots s_{i-1}s_i,v]u_i\sim  [s_1 \cdots s_i,v],
\end{align*}
where the last equality holds by (ii), since $[s_1 \cdots s_i,v] \in [G_S,G_S]$. 

Setting $i=k$, we conclude $\bar{w}=[\bar{u},\bar{v}] \in \langle S \rangle$. Since $\rho_\X(\bar{w})=1$, $\rho_S(\bar{w})=1$. Then $\bar{w}=1$ since $S$ is retractive. 
\end{proof}

Earlier versions of this article claimed that the converse of Theorem~\ref{thm:injtest} was clearly true, but in fact it is false, as illustrated by Example 5.11 of \cite{CF19}. 

In order to apply any of these results, it is necessary that $\{y\}$ is retractive for every $y \in Y$. Of course this is difficult to detect in general. For our applications, the condition is guaranteed by Corollary~\ref{singletons} below, as illustrated in Corollary~\ref{sing2}. We will use the following slightly simplified terminology.

\begin{definition} Suppose $\{y\}$ is retractive for every $y \in Y$. An antichain $\X \subseteq 2^Y$ is a {\em retractive family} if all elements of \X\ and all their finite intersections are retractive.
\end{definition}

If \X\ is a retractive family, the atomic lattice $\sS=\bar{\X}$ consists of retractive sets, and $\X=\max(\sS)$.

We close this section with some further sufficient conditions for retractiveness useful for arrangement-like groups.

\begin{proposition} Let $Y$ be a finite set of generators for $G$ and $S \subseteq Y$. Suppose $\langle S \rangle$ is a homomorphic image of $G_S$, and $G_S$ is a Hopfian group. Then $S$ is retractive.
\label{hopf}
\end{proposition}

\begin{proof} Let $\phi \colon G_S \to \langle S \rangle$ be a surjection. The composite $\rho_S\circ \phi \colon G_S \to G_S$ is a surjection, hence an isomorphism since $G_S$ is Hopfian. Then $\rho_S$ is injective since $\phi$ is surjective. 
\end{proof}

\begin{corollary}  If $G_S$ is a free group of rank $|S|$, then the subset $S$  of $Y$ is retractive. In particular, if $F=F(Y)$ is the free group with basis $Y$, then every family $\sS \subseteq 2^Y$ is retractive.
\label{prop:freeret}
\end{corollary}

\begin{corollary} Let $y \in Y$. Suppose the image of $y$ in $G/[G,G]$ has infinite order. Then $\{y\}$ is retractive. In particular, if $G/[G,G]$ is free abelian of rank $|Y|$, then $\{y\}$ is retractive for every $y \in Y$.
\label{singletons}
\end{corollary}

\begin{corollary} Suppose $|S|=r$ and $G_S \cong \Z \times F_{r-1}$. Suppose $\langle S \rangle$ is generated by some $r$ of its elements, one of which is central in $\langle S \rangle$. Then $S$ is retractive.
\label{cor:central2}
\end{corollary}

\begin{proof} The hypothesis implies that $\langle S \rangle$ is a homomorphic image of $\Z \times F_{r-1}$. Since $\Z \times F_{r-1}$ is residually finite and finitely-generated, it is Hopfian \cite{MKS76}. Hence $S$ is retractive by Proposition~\ref{hopf}. 
\end{proof}

\begin{corollary} Suppose $G$ is a group with set of generators $Y$, and $S \subseteq Y$. Let $\Delta$ be the product of the elements of $S$, in some order, and let $R_S=\{[\Delta, y] \mid y \in S \}$. Suppose $G$ has a presentation $\langle Y \mid R \rangle$ where $R_S \subseteq R$, and $G_S = \langle S \mid R_S \rangle$. Then $S$ is retractive.
\label{conjfree}
\end{corollary}

\begin{proof} We have that $G_S \cong \Z \times F_{r-1}$, with $r=|S|$, from the given presentation. By the hypothesis on $G$, $\Delta$ is central in $\langle S \rangle$, and by the definition of $\Delta$, $\langle S \rangle$ is generated by $\Delta$ and $r-1$ elements of $S$, for instance the first $r-1$ factors of $\Delta$. Then $S$ is retractive by the preceding result. 
\end{proof}

\begin{remark} A special case of Corollary \ref{conjfree} subsumes the family of cyclically-presented hyperplane groups from \cite{Fried16}. Suppose $G$ has a presentation with generators $Y=\{y_1 \ldots, y_n\}$ and relations of the form 
\begin{equation}
y_{i_t} \cdots \, y_{i_1} = y_{i_{t-1}} \cdots \, y_{i_1}y_{i_t} = \ldots\ldots = y_{i_1} y_{i_t}\cdots \, y_{i_2},
\tag{$\dag$}
\label{cycpres}
\end{equation}
for a collection \X\ of subsets $S=\{i_1, \ldots, i_t\} \in 2^Y$, whose pairwise intersections have 
at most one element. 
Setting $\Delta_S=y_{i_t} \cdots \, y_{i_1}$, the set of relations \eqref{cycpres} is equivalent to the set $\{[\Delta_S, y_{i_j}] \mid 1 \leq j \leq t\}$ of Corollary~\ref{conjfree}. 
Then \X\ is a retractive family. If the order of factors in $\Delta_S$ is compatible with the given linear ordering of $Y$ for each $S \in \X$, then $G$ is a conjugation-free, cyclically-presented hyperplane group as defined in~\cite{Fried16}.
\label{rem:cycpres}
\end{remark}

\subsection{Retractive families for arrangement groups}

Let \A\ be an affine hyperplane arrangement in $\C^\ell$, with complement $M=M(\A)=\C^\ell - \bigcup_{H\in \A} H$ and group $G=G(\A)=\pi_1(M,x_0)$, $x_0 \in M$. 
Let $\LL=\LL(\A)$ be the intersection poset of \A. For $\emptyset\neq X \in \LL$, denote the complement 
\[
M(\A_X)=\C^\ell - \bigcup_{H \supseteq X} H
\] 
of $\A_X$ by $M_X$, and the inclusion $M \hookrightarrow M_X$ by $i_X$. 
Let $B_X$ be a closed ball 
centered at a generic point $p$ of $X$, with sufficiently small radius to ensure that $B_X \cap M_X \subseteq M$. Choose a point $x_X \in B_X \cap M$ and a 
path $\gamma_X$ in $M$ from $x_0$ to $x_X$ for each $X$. 

For $H \in \A$, the composite $B_H \cap M \hookrightarrow \C^\ell-H \to \C^\ell/H - H \cong \C^\times$ is a homotopy equivalence, so one has a 
natural identification of $\pi_1(B_H \cap M,x_H)$ with \Z. 
Let $\omega_H$ be a loop in $B_H \cap M$ based at $x_H$ 
representing the positive 
generator of $\pi_1(B_H \cap M,x_H)$. The element $y_H$ of $G$ represented by the loop $\gamma_H \omega_H \gamma_H^{-1}$ has the property that $(i_X)_*(y_H)=1$ in $\pi_1(M_X,x_0)$, if $H \not \in \A_X$. We call $y_H$ a {\em meridian} of $H$. 

Let $S$ be a generic affine complex line in $\C^\ell$ containing $x_0$. Then $M \cap S$ is the plane $S$ with the $|\A|$ points $S \cap H$, $H \in \A$, removed, and $\pi_1(M \cap S,x_0)$ is a free group of rank $|\A|$.
By the Zariski-Lefschetz hyperplane section theorem of Hamm and L\^e, \cite{HaLe73}, the inclusion-induced homomorphism $\pi_1(M \cap S, x_0) \to \pi_1(M,x_0)$ is surjective. 
By construction, for each $H \in \A$, $y_H$ is the image of a conjugate of the generator of $\pi_1(M \cap S, x_0)$ corresponding to $S \cap H$. 

\begin{definition} A {\em standard set of generators} of $G$ is a generating set of $G$ consisting of one meridian $y_H$ for each hyperplane $H \in \A$.
\end{definition}

Braid monodromy presentations and the minimal presentations obtained from 
Randell-Arvola presentations of $\pi_1(M,x_0)$ have standard set of generators \cite{CS2,Ar3,OT92,F3}. 

For any arrangement, it is well known that $H_1(G;\Z)=H_1(M;\Z)$ is free abelian, generated by the homology classes of the meridians. 
Corollary~\ref{singletons} yields the following.

\begin{corollary} If $Y$ is a standard generating set of $G$, then $\{y\}$ is retractive for all $y \in Y$.
\label{sing2}
\end{corollary}

Let $Y=\{y_H \mid H \in \A\}$ be a fixed standard set of generators. For $\B \subseteq \A$, set $S_\B=\{y_H \mid H \in \B\}\subseteq Y$. The correspondence $\B \mapsto S_\B$ identifies $2^\A$ with $2^Y$, and embeds the poset of flats $\L=\L(\A)$ as an atomic lattice in $2^Y$. If $\B=\A_X$ for some $X$ in the intersection poset $\LL=\LL(\A)$ we write $S_\B=S_X$. The correspondence $X \mapsto S_X$ embeds $\LL$ as an atomic lattice in $2^Y$. A subset of $\LL$ is an atomic lattice or antichain if and only if the corresponding family in $2^Y$ is. For $\B \subseteq \A$ write $G_\B$ for $G_{S_\B}$ and $\rho_\B=\rho_{S_\B}$, and for $\X \subseteq 2^\A$ write 
\[
\rho_\X = \prod_{\B \in \X} \rho_\B \colon G \to \prod_{\B \in \X} G_\B.
\]

If \A\ is a central arrangement of rank two, $S_X$ is retractive for every $X \in \LL$, by Corollary~\ref{singletons}, and there are no other retractive subsets of $2^Y$. In general whether $S_\B$ is retractive depends on $Y$; Theorem~\ref{thm:localret} below establishes a sufficient condition. Similar considerations occupy Section 4 of \cite{DSY16}. 

The following result is perhaps well-known - see \cite{Fu82}; we provide a detailed proof.
\begin{lemma} Let $X \in \LL$, and let $i=i_X \colon M \to M_X$.
Then $i_*$ induces an isomorphism $G_X \to \pi_1(M_X,x_0)$.
\label{induce}
\end{lemma}

\begin{proof} Since $i_*(y_H)=1$ for $H \not \in \A_X$, 
$i_*$ induces a well-defined map $\alpha \colon G_X \to \pi_1(M_X,x_0)$, satisfying $i_*=\alpha \circ \rho_X$.  We claim $\alpha$ is an isomorphism.

First assume $x_0 \in B_X$. Let $\hat{r} \colon \C^\ell \to \C^\ell$ be the radial retraction onto $B_X$. The image of $M_X$ under $\hat{r}$ is $B_X \cap M$, hence $\hat{r}$ restricts to a map $r \colon M_X \to M$, and the composite $i \circ r \colon M_X \to M_X$ is a homotopy equivalence. Let $\beta=\rho_X \circ r_* \colon \pi_1(M_X,x_0) \to G_X$.

We have the following commutative diagram of groups and homomorphisms.
\[
\xymatrix@1{
G \ar[rr]^-{{\id}} \ar[dr]^-{\rho_X} & &G\ar[dd]^-{i_*} \ar[dl]_{\rho_X}\\
&G_X\ar[dr]^{\alpha}&\\
\pi_1(M_X,x_0) \ar[uu]^{r_*} \ar[rr]^{\cong} \ar[ur]^{\beta}&& \pi_1(M_X,x_0)
}
\]
We conclude $\alpha$ is surjective.

Now suppose further that $y_H \in \im\left(\pi_1(B_X \cap M,x_0) \to \pi_1(M,x_0)\right)$ for each $H \in \A_X$. Since $r$ restricts to the identity on $B_X \cap M$, we have the commutative diagram
\[
\xymatrix@1{
G \ar[rr] \ar[dr]^{i_*} \ar[dd]_{\rho_X}& &G\ar[dd]^-{\rho_X} \\
&\pi_1(M_X,x_0)\ar[ur]^{r_*}\ar[dr]^{\beta}&\\
G_X \ar[rr]^{{\id}} \ar[ur]^{\alpha}&& G_X
} 
\]
from which we conclude $\alpha$ is an isomorphism under these assumptions.

For the general case, choose a path $\gamma$ in $M$ from the point $x_X$ in $B_X \cap M$ to the base point $x_0$. For $H \in \A_X$ choose a meridian $y_H'$ of $H$ in $B_X \cap M$, based at $x_X$. For $H \in \A-\A_X$, set $y_H'=(\gamma_\#)^{-1}(y_H)$. Then 
$
Y'=\{y_H' \mid H \in \A_X \}
$
is a standard set of generators for $\pi_1(M,x_X)$, and $i_*(y_H')=1$ for all $H \in \A-\A_X$. Moreover the conditions in the special cases treated above are satisfied. Consequently, $i$ induces an isomorphism $G_{X,x_X} \to \pi_1(M_X,x_X)$, where 
\[
G_{X,x_X}= \pi_1(M,x_X)/\langle\langle y_H' \mid H \in \A-\A_X\rangle \rangle. 
\]
The isomorphism $\gamma_\# \colon \pi_1(M,x_X) \to G$ induces an isomorphism $G_{X,x_X} \to G_X$ by definition of $y_H'$ for $H \in \A-\A_X$. Then
the commutative diagram
\[
\xymatrix@1{
G_{X,x_X} \ar[rr]^-{i_*}_-{\cong}\ar[d]_-{\gamma_\#}^{\cong} && \pi_1(M_X,x_X) \ar[d]_{\cong}^-{\gamma_\#} \\
G_X\ar[rr]^-{i_*}&& \pi_1(M_X,x_0) 
} 
\]
implies $i_* \colon G_X \to \pi_1(M_X,x_0)$ is an isomorphism. 
\end{proof}

The condition on $y_H$ in Lemma~\ref{induce} is satisfied if some conjugate of $y_H$ is represented by the loop $\gamma_H \omega_H \gamma_H^{-1}$ constructed earlier. This is the case in the presentations discussed above, and will be so in all our applications.

\begin{theorem} \label{thm:localret}
Let $X \in \LL$, and let $j \colon B_X \cap M \to M$ and $i \colon M \to M_X$ be the inclusion maps. Assume that $Y=\{y_H\mid H \in \A\}$ is standard generating set and 
\begin{equation} \label{eq:localeq}
(\gamma_X)_\#(y_H) \in \im \left(\pi_1(B_X \cap M,x_X ) \xlongrightarrow{j_*} \pi_1(M,x_X)\right)\tag{$\dag\dag$} \ \text{for every} \ H \in \A_X.
\end{equation}
Then the subset $S_X=\{y_H\mid H \in \A_X\}$ of $Y$ is retractive.
\end{theorem}

\begin{proof} Write $\gamma=\gamma_X$. Since $i \circ j \colon B_X \cap M \to M_X$ is a homotopy equivalence, the restriction of $i_*\colon \pi_1(M,x_X) \to \pi_1(M_X,x_X)$ to the image of $j_* \colon \pi_1(B_X \cap M,x_X) \to \pi_1(M,x_X)$ is injective. Then \eqref{eq:localeq} implies that the restriction of $i_* \colon \pi_1(M,x_X) \to \pi_1(M_X,x_X)$ to $\gamma_\#(\langle S_X \rangle)$ is injective. Then, since $\gamma_\#$ is an isomorphism commuting with the inclusion-induced homomorphisms, the restriction of $i_* \colon G \to \pi_1(M_X,x_0)$ to $\langle S_X \rangle$
is injective. 
Then Lemma~\ref{induce} implies $S_X$ is retractive. 
\end{proof}

\begin{definition} Let $\X \subseteq 2^\A$ be an antichain. A standard set of generators $Y$ of $G$ is {\em adapted to} $\X$ if the subset $S_\B$ of $Y$ is retractive, for each $\B \in \bar{\X}$, the atomic lattice generated by $\X$ in $2^\A$. 
\end{definition}

\begin{remark} If \X\ consists of rank two flats of \A, then distinct elements of \X\ are disjoint or have one element in common. Then $Y$ is adapted to \X\ if and only if $S_\B$ is retractive for each $\B \in \X$, by Corollary~\ref{sing2}.
\label{adaptation}
\end{remark}

Other retractive subsets of $2^Y$ arise from families of parallel hyperplanes in \A. 

\begin{lemma} If $\PP$ is a set of pairwise disjoint hyperplanes in \A\ and $i \colon M \to M(\PP)$ is the inclusion, then $\pi_1(M(\PP),x_0)$ is a free group with free basis $\{i_*(y_H) \mid H \in \PP\}$.
\end{lemma}

\begin{proof} Since the hyperplanes in \PP\ are mutually disjoint, they have a common linear complement $C \cong \C$. The inclusion $C \cap M(\PP) \to M(\PP)$ is a homotopy equivalence, with homotopy inverse given by restricting the linear projection onto $C$, and $C \cap M(\PP)$ is the complement in $C$ of the finite set $\{H \cap C \mid H \in \PP\}$. Then $\pi_1(M(\PP),x_0)$ is a free group. Because $y_H$ is a meridian of $H$ for $H \in \PP$, the set $\{i_*(y_H) \mid H \in \PP\}$ is a free basis of $\pi_1(M(\PP),x_0)$, by \cite[Cor.~3.5.1]{MKS76}.  
\end{proof}

\begin{corollary} Let $\PP$ be a set of pairwise disjoint hyperplanes in \A. Let $i \colon M \to M(\PP)$ be the inclusion map. Suppose $i_*(y_H)=1$ for all $H \in \A-\PP$. Then
\begin{enumerate}
\item $i$ induces an isomorphism $G_{\PP} \to \pi_1(M(\PP),x_0)$, and
\item the subset $S_\PP$ of $Y$ is retractive.
\end{enumerate}
\label{parret}							
\end{corollary}

\begin{proof} The first statement is an immediate consequence of the preceding lemma, and the second statement then follows from Corollary~\ref{prop:freeret}. 
\end{proof}

A maximal subset of pairwise disjoint hyperplanes will be called a {\em parallel class}. The hyperplanes in a parallel class are those whose closures contain a fixed codimension two subspace of the hyperplane at infinity in projective space. 

\begin{corollary} Let $\X=\X_0 \cup \X_\infty$, where $\X_0 \subseteq \L$ is an antichain and $\X_\infty \subseteq 2^\A$ is a set of parallel classes. Let $Y$ be a standard set of generators of $G$ adapted to $\X_0$. Then $\ker(\rho_\X)$ is generated by monic commutators whose support is transverse to~\X.
\label{cor:monicker}
\end{corollary}

\begin{proof} Let $\bar{\X}_0$ and $\bar{\X}_\infty$ denote the atomic lattices generated by $\X_0$ and $\X_\infty$ in $2^\A$, respectively, and set $\bar{\X}=\bar{\X}_0 \cup \bar{\X}_\infty$. Parallel classes are disjoint, and for any flat $\B \in \L$, $\B$ has at most one hyperplane in any given parallel class. Then 
$\bar\X$ is an atomic lattice, 
so $Y$ is adapted to \X\ by the assumption that $Y$ is adapted to $\X_0$ together with Corollary~\ref{parret} (ii). 
The statement is a then consequence of Corollary~\ref{thm:kernel}. 
\end{proof}

\subsection{Examples}\label{subsec:examples}
The next 
result is useful in verifying condition (ii) of Theorem~\ref{thm:injtest}.
\begin{proposition}
Let $G$ be a group with set of generators $Y$. Let $S=\{s_0,\ldots s_m\}$
and $T=\{t_0,\ldots, t_n\}$ be subsets of $Y$, with $s_0=t_0$. Assume
that:
\begin{enumerate}
\item $[s_i,t_j]=1$ for $1\leq i \leq m$ and $1\leq j\leq n$, and
\item $[t_0 \cdots t_n,t_j]=1$ for $1 \leq j \leq n$.
\end{enumerate}
Then $[s_i, [\langle T\rangle,\langle T\rangle]]=1$ for $1\leq i \leq
m$ and  $[t_i, [\langle S\rangle,\langle S\rangle]]=1$ for $1\leq i \leq n$.
\label{prop:simplify}
\end{proposition}

\begin{proof} Let $t=t_1 \cdots t_n$. The set $\{\{t_0t ,t_1, \ldots, t_n\}$ generates the subgroup $\langle T \rangle$ of $G$.  Applying \cite[Problem 2.1.8]{MKS76} to this generating set, we see that 
elements of 
$[\langle T \rangle,\langle T \rangle]$ may be expressed as (commutators of) words in $T-\{t_0\}$. The first assertion then follows from (i).

For the second assertion, note that, since $s_0=t_0$, $[t_i,s_0t]=1$ by (ii), and $t_i^{s_0}=t_i^{t^{-1}}$ for each $i$, $1\le i\le n$. Using this, if $u$ and $v$ are words in $S-\{s_0\}$, then 
\[
t_i^{us_0v}= t_i^{s_0v} = t_i^{t^{-1}v}= t_i^{vt^{-1}}= t_i^{t^{-1}}= t_i^{s_0}.
\]
Consequently, for any $w \in \langle S\rangle$, we have $t_i^w=t_i^{s_0^m}$, where $m$ is the exponent sum of $s_0$ in $w$. In particular, $t_i^w=t_i$ for $w\in  [\langle S\rangle,\langle S\rangle]$, and $[t_i,w]=1$ as needed. 
\end{proof}

\begin{example}
Let \A\ be the arrangement of lines in $\C^2$ with defining equations $u=0$, $u=1$, $v=0$, $v=-1$, and $u+2v=0$, see Figure~\ref{fig:arvola1}(a).
Using the Randell algorithm \cite{R3,F3}, the fundamental group $G$ of the complement of \A\ has a presentation with generators $y_i$, $1\le i \le 5$, corresponding to the lines, and relations $[y_1,y_4]$, $ [y_2,y_4]$, $[y_2,y_3]$, $[y_2,y_5]$, $[y_4,y_5]$, and $[y_1y_3y_5,y_i]$ for $i=1,3,5$. With $\X_0=\{\{H_1,H_3,H_5\}\}$, the standard generating set $Y=\{y_1, y_2, y_3, y_4, y_5\}$ is adapted to $\X_0$ by Proposition~\ref{conjfree}. Let $\X_\infty=\{\{H_1,H_2\},\{H_3,H_4\}\}$ and $\X=\X_0\cup \X_\infty$. 

\begin{figure}[h]
\begin{subfigure}{.49 \textwidth}
\centering
\includegraphics[scale=.2]{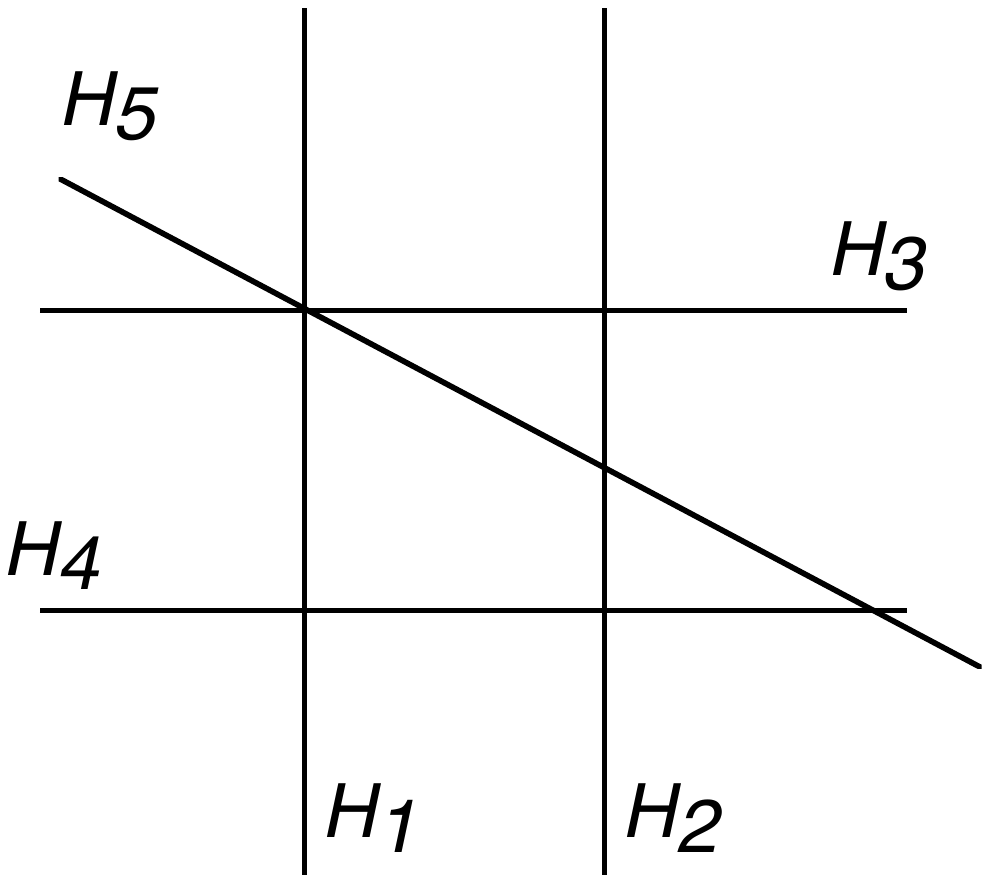}
\caption{Example~\ref{ex:arvola}}
\end{subfigure}
\begin{subfigure}{.49 \textwidth}
\centering
{\includegraphics[scale=.2]{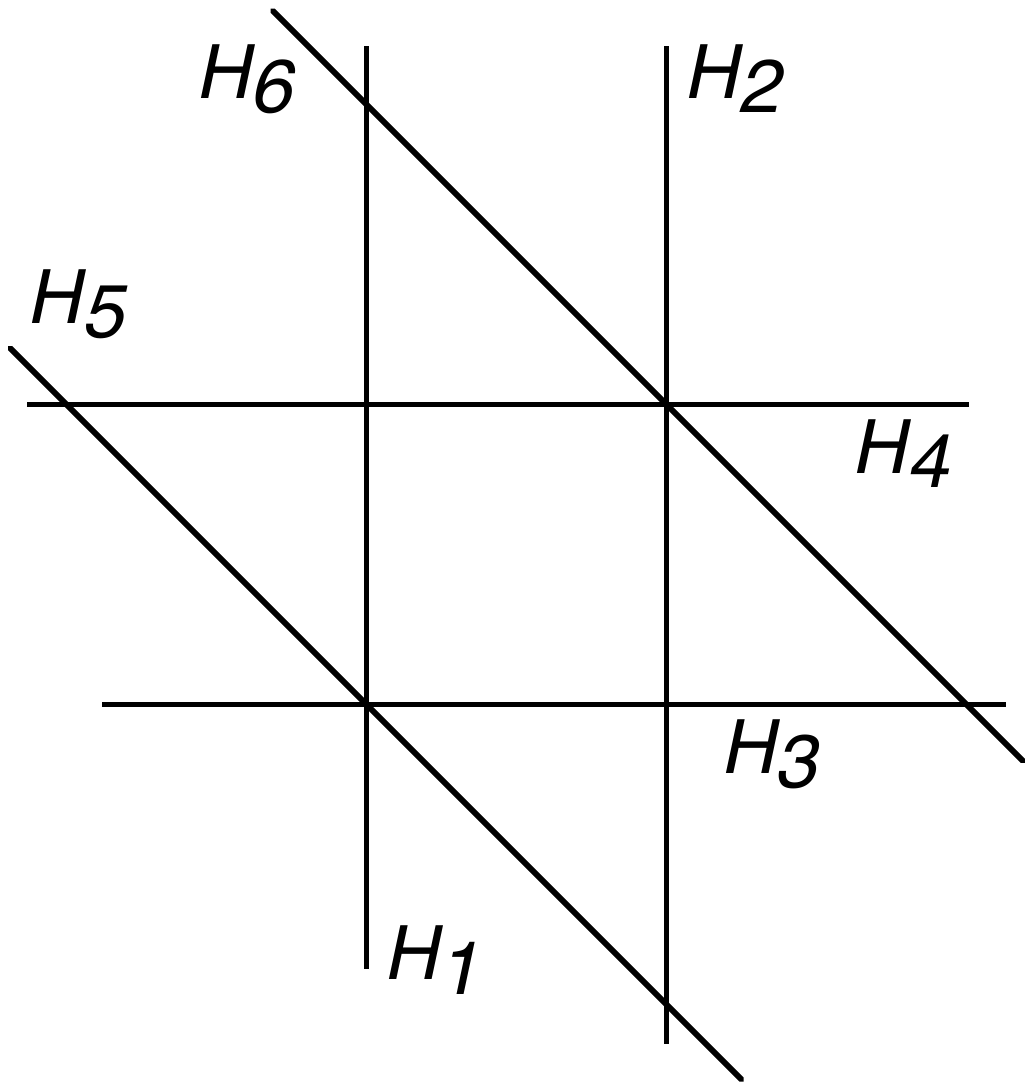}}
\caption{Example~\ref{ex:kohno}}
\end{subfigure}
\caption{Arrangements with residually free groups}
\label{fig:arvola1}
\end{figure}

We verify the conditions of Theorem~\ref{thm:injtest}. 
Condition (i) holds because $[y_i,y_j]=1$ for $(i,j)=(1,4), (2,3), (2,4)$, or $(2,5)$. 
Condition (ii) may be obtained from Proposition~\ref{prop:simplify} by first setting 
$\{s_0,s_1,s_2\}=\{y_1,y_3,y_5\}$ and $\{t_0,t_1\}=\{y_1,y_2\}$, and second setting 
$\{s_0,s_1,s_2\}=\{y_3,y_5,y_1\}$ and $\{t_0,t_1\}=\{y_3,y_4\}$. 
So  
$\rho_\X \colon G \to \prod_{\S \in \X} G_\S \cong (\Z \times F_2) \times F_2 \times F_2$ 
is injective.
\label{ex:arvola}
\end{example}

\begin{example}
Let \A\ be the arrangement of lines in $\C^2$ with defining equations $u=0,u=1, u+v=0, u+v=2, v=0$, and $v=1$, illustrated in Figure~\ref{fig:arvola1}(b). The group $G$ has generators $y_i$, $1\le i\le 6$, corresponding to the lines, with relations 
\[
 [y_2,y_5], \ [y_2,y_3], \ [y_6,y_3], \ [y_4,y_5^{y_1}], \ [y_1,y_4],  \ [y_1,y_6^{y_2}], 
\]
along with $[y_1y_3y_5,y_i]$ for $i=1,3,5$, and $[y_2y_4y_6,y_i]$ for $i=2,4,6$.

If
$
\X_0=\{\{H_1,H_3,H_5\},\{H_2,H_4,H_6\}\}
$, then $Y=\{y_1,\ldots,y_6\}$ is adapted to $\X_0$ by Corollary~\ref{conjfree}.
Let
$
\X_\infty=\{\{H_1,H_2\},\{H_3,H_4\},\{H_5,H_6\}\}
$
 and $\X=\X_0\cup \X_\infty$. 

Since $[y_4,y_5^{y_1}]=[y_1,y_4]=1$, we have $[y_4,y_5]=1$. Since $[y_2y_4,y_6]=1$, $y_6^{y_2}=y_6^{y_4^{-1}}$. Then from $[y_1,y_6^{y_2}]=1$ and $[y_1,y_4]=1$ we conclude $[y_1,y_6]=1$. Along with the given relations, this confirms that condition (i) of Theorem~\ref{thm:injtest} holds.  Condition (ii) of Theorem~\ref{thm:injtest} can be verified using Proposition~\ref{prop:simplify} as in the preceding example. It follows that $\rho_\X$ is injective.
\label{ex:kohno}
\end{example}

\begin{example} Let $\A=\D_3$ be the arrangement of type $D_3$ as in Example~\ref{ex:d3}. The complement of $\A$ is homeomorphic to the complement $PB_4$ of the rank three braid arrangement, and $G=G(\A)=P_4$ is the four strand pure braid group. As shown in Example \ref{ex:brunnian}, \A\ has four rank two flats of size three, and the kernel of the resulting product homomorphism $\rho_\X \colon P_4 \to \prod_{\S\in \X} G_\S$ is isomorphic to the group of Brunnian braids on four strands. In particular, $\rho_\X$ is not injective. 

The entire arrangement \A\ itself supports the generating function $f\colon M(\A) \to PB_3$, $f(z_1,z_2,z_3)=(z_1^2,z_2^2,z_3^2)$, noted in Example~\ref{ex:d3}.  
One can thus consider the product homomorphism $\rho_\X \times f_* \colon P_4 \to \prod_{\S\in \X} G_\S \times P_3$ given by all five generating sets on \A. The target is a product of free groups. There are nontrivial Brunnian braids in the kernel of $f_*$, hence $\rho_\X \times f_*$ is not injective. In fact, certain nontrivial Brunnian braids are in the kernel of any homomorphism from $P_4$ to a free group, hence $P_4$ is not residually free, see \cite{CFR11}.
\end{example}

\section{Arrangement groups and right-angled Artin groups}
A right-angled Artin group (RAAG) is a group that has a finite presentation in which all relations are commutators of two generators. The group is determined by the undirected graph whose vertices are the generators, with edges connecting pairs of commuting generators. This family includes products of free groups. 

Suppose \A\ is an affine arrangement, and $\S$ is a rank two flat, or a parallel class in \A. Then $G_\S$ is isomorphic to $\Z \times F_{r-1}$, or to $F_r$, where $r=|\S|$. Consequently, if $\X\subseteq 2^\A$ is a set of rank two flats and parallel classes, the target of 
\[
\rho_\X \colon G \to \prod_{\S\in \X} G_\S
\]
is a product of free groups, hence is a right-angled Artin group. 

In this setting it will be more efficient to replace factors $G_\S$ isomorphic to $\Z \times F_{r-1}$ in this product by the free group $\bar{G}_\S=G_\S/\Z$. For a parallel class $\S$, write $\bar{G}_\S=G_\S$. 
In this section we show that the image of the homomorphism 
\[
G \xrightarrow{\ \rho_\X\ } \prod_{\S \in \X} G_\S \to \prod_{\S \in \X} \bar{G}_\S
\]
is normal, and identify the cokernel with the first integral cohomology of the incidence graph of \X. We also show that projecting to $\prod_{\S\in \X}\bar{G}_\S$ does not affect the kernel of $\rho_\X$. As a result we are able to realize some arrangement groups as subgroups of right-angled Artin groups, drawing conclusions about their qualitative and homological finiteness properties.

\subsection{The cokernel of \texorpdfstring{$\rho_\X$}{rX}} 
Let $G$ be a group with finite set of generators $Y=\{y_1,\ldots, y_n\}$. Assume that $z=y_1 y_2\cdots y_n$ is central in $G$. For example, if $G$ is the fundamental group of the complement of a central arrangement, the braid monodromy presentation yields a set of generators satisfying this condition. We consider subsets of $Y$ to be linearly ordered using the given labelling. Let $\X=\{S_1,\ldots, S_m\} \subseteq 2^Y$. For $1\leq i \leq m$ and $1\leq j\leq n$, let $y_{ij}$ denote the image of $y_j$ in $G_{S_i}$. Then $G_{S_i}$ is generated by $\{y_{ij} \mid j \in S_i\}$, and $y_{ij}=1$ if $j \not \in S_i$. Since $z$ is central in $G$, $z_i=y_{i1}\cdots y_{in}$ is central in $G_{S_i}$, for each $i$.
Let $\bar{G}=G/\langle z\rangle$ and $\bar{G}_{S_i}=G_{S_i}/\langle z_i\rangle$. Let $\rho_\X =\prod \rho_{S_i} \colon G \to \prod_{i=1}^m G_{S_i}$ be the associated homomorphism. The image of $y_j$ under $\rho_\X$ is $\prod_{r=1}^m y_{rj}$. Since $\rho_{S_i}(z)=z_i$, $\rho_\X$ induces a well-defined homomorphism $\bar{\rho}_\X \colon \bar{G} \to \prod_{i=1}^m \bar{G}_{S_i}$. 

Assume further that $\bar{G}_{S_i}$ is a free group of rank $|S_i|-1$, so that the images of any $|S_i|-1$ of the elements $y_{ij}$, $j \in S_i$ form a free basis. Then $\prod_{i=1}^m \bar{G}_{S_i}$ is a right-angled Artin group, whose graph is the complete multipartite graph with parts of sizes $|S_1|-1, \ldots, |S_m|-1$. Viewing $\bar{G}_S$ as a subgroup of $\prod_{S\in \X} \bar{G}_S$, we have $[y_{ij},y_{rk}]=1$ for $i\neq r$. 

Let $A=\prod_{S\in \X} G_S$ and $\bar{A}=\prod_{S\in \X} \bar{G}_S$. 
We have the commutative diagram below.
\[
\xymatrix{
G \ar[r]^{\rho_\X} \ar[d] & A \ar[d]\\
\bar{G} \ar[r]^{\bar{\rho}_\X} & \bar{A}
}
\]

In our application to arrangement groups, $G=G(\A)$ for a central arrangement \A, $\bar{G}$ is the fundamental group of the projectivized complement $\bar{M}$, \X\ is a set of rank two flats of $\A$,
and $\bar{\rho}_\X \colon \bar{G} \to \bar{A}$ is the product of the homomorphisms induced by inclusions of projectivized complements. Rank two flats are not retractive in the projective setting, so our analysis of the kernel does not apply to $\bar{\rho}_\X$, only to $\rho_\X$.

\begin{proposition} Suppose $|S_i \cap S_r|\leq 1$ for all $1\leq i, r \leq m$.
Then the image of $\rho_\X$ is a normal subgroup of $A$. \label{thm:normal}
\label{normal}
\end{proposition}

\begin{proof} Fix integers $i,j,k$ with $1\leq i \leq m$ and $1\leq j,k\leq n$, and consider the conjugate $\rho_\X(y_k)^{y_{ij}}$. If $S_i$ does not contain both $j$ and $k$, then $\rho_\X(y_k)^{y_{ij}}=\rho_\X(y_k)$. Suppose $S_i$ contains both $j$ and $k$. If $r\neq i$ then $S_r$ does not contain both $j$ and $k$, so $y_{rk}^{y_{rj}}=y_{rk}=y_{rk}^{y_{ij}}$. If $r=i$ then $y_{rk}^{y_{rj}}=y_{rk}^{y_{ij}}$.
Then we have
\[
\rho_\X(y_k)^{y_{ij}}  = (\prod_{r=1}^m y_{rk})^{y_{ij}}
= \prod_{r=1}^m y_{rk}^{y_{ij}}
= \prod_{r=1}^m y_{rk}^{y_{rj}}
=(\prod_r y_{rk})^{\prod_r y_{rj}}
=\rho_\X(y_k)^{\rho_\X(y_j)}.
\]
So, in either case, $\rho_\X(y_k)^{y_{ij}}$ lies in the image of $\rho_\X$. 
\end{proof}

Let $\Lambda_\X$ be the bipartite graph with vertex set $\X \cup Y$ and edges $\{S_i,y_j\}$ for $y_j\in S_i$. 
The hypothesis of Proposition~\ref{normal} is equivalent to the condition that $\Lambda_\X$ contain no cycles of length four. For the remainder of this subsection assume that \X\ has this property. 
Denote the image of ${\rho}_\X$ by $N_\X$ and the image of $\bar{\rho}_\X$ by $\image_\X$.

\begin{corollary} The image $\image_\X$ of $\bar{\rho}_\X$ is a normal subgroup of $\bar{A}$. 
\label{cor:normal}
\end{corollary}

\begin{proof} The surjection $A\to \bar{A}$ maps $N_\X$ onto $\image_\X$. 
\end{proof}

\begin{proposition} The cokernel of  $\rho_\X \colon G \to A$ is abelian.
\label{prop:coker}
\end{proposition}

\begin{proof} As in proof of Proposition~\ref{normal}, we observe that $[y_{ij},y_{ik}]=[\rho_\X(y_j),\rho_\X(y_k)]$ if $S_i$ contains both $j$ and $k$, by the assumption on \X, and is trivial otherwise. Since $[y_{ij},{y_{rk}}]=1$ if $r\neq i$ this shows that $A/N_\X$ is abelian.  
\end{proof}

\begin{corollary} The cokernel of  $\bar{\rho}_\X\colon \bar{G} \to \bar{A}$ is abelian.
\end{corollary}

\begin{proof} The group $\bar{A}/\image_\X$ is a quotient of $A/N_\X$. 
\end{proof}

The fact that $\bar{A}/\image_\X$ is abelian can also be deduced directly from the normality of $\image_\X$ and the fact that it surjects onto each factor of $\bar{A}$, by a result of \cite{BriMi09}.

We denote the abelianization of a group or homomorphism by appending the subscript $\ab$. So, for example, $A_{\ab}=A/[A,A]$. 

Assume that $G_{\ab}$ and $A_{\ab}$ are free abelian, with free 
bases given by the images of $y_1, \ldots, y_n$ and $y_{ij}$, $j \in S_i$, respectively. 
This implies in particular that the central elements $z$ and $z_i, 1\leq i \leq m$ have infinite order. 
Again, this hypothesis holds if $G=G(\A)$ for a central arrangement \A\ and $\{y_1,\ldots, y_n\}$ is a standard set of generators of $G$, ordered appropriately.

Denote the images of $y_k$ and $y_{ij}$ in $G_{\ab}$ and $A_{\ab}$ by $b_k$ and $b_{ij}$, respectively. (So $b_{ij}=0$ if $y_j \not \in S_i$.) Then $\bar{G}_{\ab}$ is the quotient of $G_{\ab}$ by the subgroup generated by $\sum_{k=1}^n b_k$, and $\bar{A}_{\ab}$ is the quotient of $A_{\ab}$ by the subgroup generated by 
\[
\left\{\sum_{j=1}^n b_{ij} \mid 1\leq i\leq m\right\}. 
\]
The latter subgroup will be denoted by $J$. 

\begin{corollary} The homomorphisms $\bar{\rho}_\X \colon \bar{G} \to \bar{A}$ and 
$(\bar{\rho}_\X)_{\ab}\colon \bar{G}_{\ab} \to \bar{A}_{\ab}$ have isomorphic cokernels.
\label{cor:abcoker}
\end{corollary}

\begin{proof}
Since $\bar{A}/\image_\X$ is abelian, it is a quotient of $\bar{A}_{\ab}$, and the kernel of the induced map $\bar{A}_{\ab} \to \bar{A}/\image_\X$ is easily seen to be $(\bar{\rho}_\X)_{\ab}(\bar{G}_{\ab})$.
\end{proof}

\begin{theorem} The cokernel of $\bar{\rho}_\X \colon \bar{G} \to \bar{A}$ is isomorphic to the integral simplicial cohomology group $H^1(\Lambda_\X,\Z)$ of the 1-dimensional simplicial complex $\Lambda_\X$.
\end{theorem}
\label{thm:DM}
\begin{proof} By the preceding result we need only identify the cokernel of $(\bar{\rho}_X)_{\ab}$. The group $A_{\ab}$ is naturally identified with the additive group of integer edge-labelings of $\Lambda_\X$, that is, with the group $C^1(\Lambda_\X,\Z)$ of integral simplicial cochains on the simplicial complex $\Lambda_\X$. The generator $b_{ij}$ corresponds to the indicator function on the edge $\{S_i,y_j\}$. The generator $\sum_{j=1}^n b_{ij}$ of the subgroup $J$ defined above is precisely the coboundary of the vertex $S_i$ of $\Lambda_\X$, considered as a 0-cochain on $\Lambda_\X$, up to sign. Similarly, the image $(\rho_\X)_{\ab}(b_j)=\sum_{i=1}^m b_{ij}$ of $b_j \in G_{\ab}$ is the coboundary $\delta(y_j)$ of the $0$-cochain $y_j$, up to sign. (Recall $b_{ij}=0$ if $y_j \not \in S_i$.)
Denote the subgroup generated by these elements by $I$. We have $\bar{A}_{\ab} = A_{\ab}/J$ and $\image_\X = I+J/J$, so that $\bar{A}_{\ab}/\image_\X$ is isomorphic to $\bar{A}_{\ab}/I+J$. The latter group is the quotient of $C^1(\Lambda_\X,\Z)$ by the image of $\delta \colon C^0(\Lambda_\X, \Z) \to C^1(\Lambda_\X,\Z)$; since $\Lambda_\X$ is one-dimensional, this quotient is $H^1(\Lambda_\X,\Z)$.
\end{proof}

\begin{corollary} The cokernel of $\bar{\rho}_\X \colon \bar{G} \to \bar{A}$ is free abelian of rank 
\[
\sum_{S\in \X} |S| -n -m+c,
\]
where $n=|Y|$, $m=|\X|$, and $c$ is the number of components of $\Lambda_\X$.
\label{cokernel}
\end{corollary}

\begin{proof} The first statement is immediate from Corollary~\ref{cor:abcoker} and the preceding theorem. The rank formula follows from a simple Euler characteristic calculation. 
\end{proof}

\subsection{Injectivity of \texorpdfstring{$\bar{\rho}_\X$}{}}
The results of the Section~3  
do not apply directly in the projective setting because the projection $\bar G \to \bar{G}_S$ need not split. At the same time, it is problematical to apply our injectivity criteria to central rank three arrangements. In this subsection we resolve these issues.

\begin{proposition} Suppose $\Lambda_\X$ is connected. Then the kernel of $\rho_\X$ projects isomorphically onto the kernel of $\bar{\rho}_\X$.
\label{prop:robar}
\end{proposition}

\begin{proof} Denote the projections $G\to \bar{G}$ and $A\to \bar{A}$ by $p$ and $q$ respectively. The kernel of $p$ is generated by the central element $y_1\cdots y_n$, hence intersects $\ker(\rho_\X)$ trivially. Thus $\ker(\rho_\X)$ injects into $\ker(\bar{\rho}_\X)$. 

If $\delta=\rho_\X(y_1\cdots y_n)$, then $\delta \in \ker(q)$. To show that $\ker(\rho_\X)$ maps onto $\ker(\bar{\rho_\X})$ it suffices to show that $\rho_\X(G)\cap \ker(q)=\langle \delta \rangle$. The image of $\rho_\X$ is generated by $\{\prod_i y_{ij} \mid 1\leq j\leq n\}$. The kernel of $q$ is the free abelian group with basis $\{\prod_j y_{ij} \mid 1\leq i \leq m\}$.  Suppose $a\in \rho_\X(G)\cap \ker(q)$, and let $g\in G$ with $a=\rho_\X(g)$. Since $a\in \ker(q)$, we may write $a=\prod_i(\prod_j y_{ij})^{k_i}$ for some integers $k_i, 1\leq i \leq m$. Note that $k_i$ is equal to the exponent sum of $y_j$ in $g$, for any $y_j\in S_i$. 

Replacing $a$ by $a'=a\delta^{-k_1}$ we may assume $k_1=0$. This implies the exponent sum of $y_j$ is zero, for any $y_j\in S_1$. For $2\leq i\leq m$, choose a path $(S_1, y_{i_1}, S_{i_2}, y_{i_2}, \ldots, y_{i_k}, S_i)$ in $\Lambda_\X$ from $S_1$ to $S_i$. The exponent sum of $y_{i_1}$ in $g$ is zero, which implies $k_{i_2}=0$ by the observation above. Then the exponent sum of $y_{i_2}$ in $g$ is 0. Then $k_{i_3}=0$.  Continuing in this way we conclude that $k_i=0$. Thus $a'=1$, so $ a=\delta^{k_1} \in \langle \delta \rangle$. 
\end{proof}

Finally we adapt the preceding result to groups of affine arrangements, which will 
facilitate checking 
the conditions of Theorem~\ref{thm:injtest} in our examples. Let $\widehat{G}$ be the subgroup of $G$ generated by $\{y_1,\ldots, y_{n-1}\}$. Assume that $G \cong \widehat{G}\times \langle z\rangle$. In particular, $\widehat{G}$ is isomorphic to $\bar{G}$. Similarly, for $1\leq i \leq m$, let $\widehat{G}_{S_i}$ be the subgroup of $G_{S_i}$ generated by $\{y_{ij} \mid 1\leq j <n\}$. Then $\widehat{G}_{S_i}=G_{S_i}$ if $y_n\not \in S_i$, and $\widehat{G}_{S_i}$ is a free group of rank $|S_i|-1$ isomorphic to $\bar{G}_{S_i}$ if $y_n \in S_i$. 
Note that 
$G_{S_i} \cong \langle z_i\rangle  \times \widehat{G}_{S_i}$ if $n \in S_i$. Write $\widehat{A}=\prod_{i=1}^m \widehat{G}_{S_i}$, and observe that $\rho_\X(\widehat{G})\subseteq \widehat{A}$. Let $\widehat{\rho}_\X \colon \widehat{G} \to \widehat{A}$ be the restriction of $\rho_\X$.

If $G=G(\A)$ for a central arrangement $\A=\{H_i\}_{i=1}^n$, and $\{y_1, \ldots, y_n\}$ is a standard set of generators of $G$ with $y_1 \cdots y_n$ central, then $\widehat{G}=G(d\A)$ is the fundamental group of the decone of $\A$ (with respect to the hyperplane $H_n$), see \cite{OT92}, and all the assumptions in the previous paragraph hold. Note also, if \X\ is a set of rank two flats, then $|S_i \cap S_r|\leq 1 $ for $1 \leq i,r \leq m$, so that the results of the preceding subsection apply.

\begin{corollary}  
If $\Lambda_\X$ is connected, $z=y_1\cdots y_n$ has infinite order in $G$, and  $y_n \not \in S_i$ for some~$i$, 
then $\bar{\rho}_\X$ is injective if and only if $\widehat{\rho}_\X$ is injective. 
\label{cor:affinj}
\end{corollary}

\begin{proof} Necessity is immediate since the restriction of $p \colon G \to \bar{G}$ to $\widehat{G}$ is an isomorphism. Suppose $\hat\rho_\X$ is injective, and let $\bar{g} \in \ker(\bar{\rho}_\X)$. By Proposition~\ref{prop:robar}, there exists $g \in \ker(\rho_\X)$ with $p(g)=\bar{g}$. Also there exists $g_0 \in \widehat{G}$ such that $p(g_0)=\bar{g}$. Then $g_0=gz^k$ for some $k \in \Z$, and $\rho_\X(g_0)=\rho_\X(z^k)$. The quotient of $A$ by the normal subgroup $\widehat{A}$ is free abelian, generated by $\{z_i \mid y_n \not \in S_i\}$, and is not trivial by hypothesis. The image of $\rho_\X(z)^k$ in this quotient is trivial, since $\rho_\X(\widehat{G})\subseteq \widehat{A}$, and this implies $k=0$. 
Thus $g=g_0$ so $\rho_\X(g_0)=\rho_\X(g)=1$. Then $g_0=1$, and $p(g_0)=\bar{g}=1$. 
\end{proof}

\begin{example} Let $\A\subset\C^3$ be the cone of the affine arrangement of Example~\ref{ex:arvola}; \A\ is defined by 
$Q=x(x-z)(x+2y)(y+z)yz$.  With the hyperplanes labelled as in Figure~\ref{fig:arvola2}(a), let 
$
\X=\{\{H_1,H_3,H_5\},\{H_1,H_2,H_6\},\{H_3,H_4,H_6\}\}
$. 
The graph $\Lambda_\X$, illustrated in Figure~\ref{fig:arvola2}(b), is connected. 

\begin{figure}[h]
\begin{subfigure}[t]{.49 \textwidth}
\centering
\includegraphics[scale=.2]{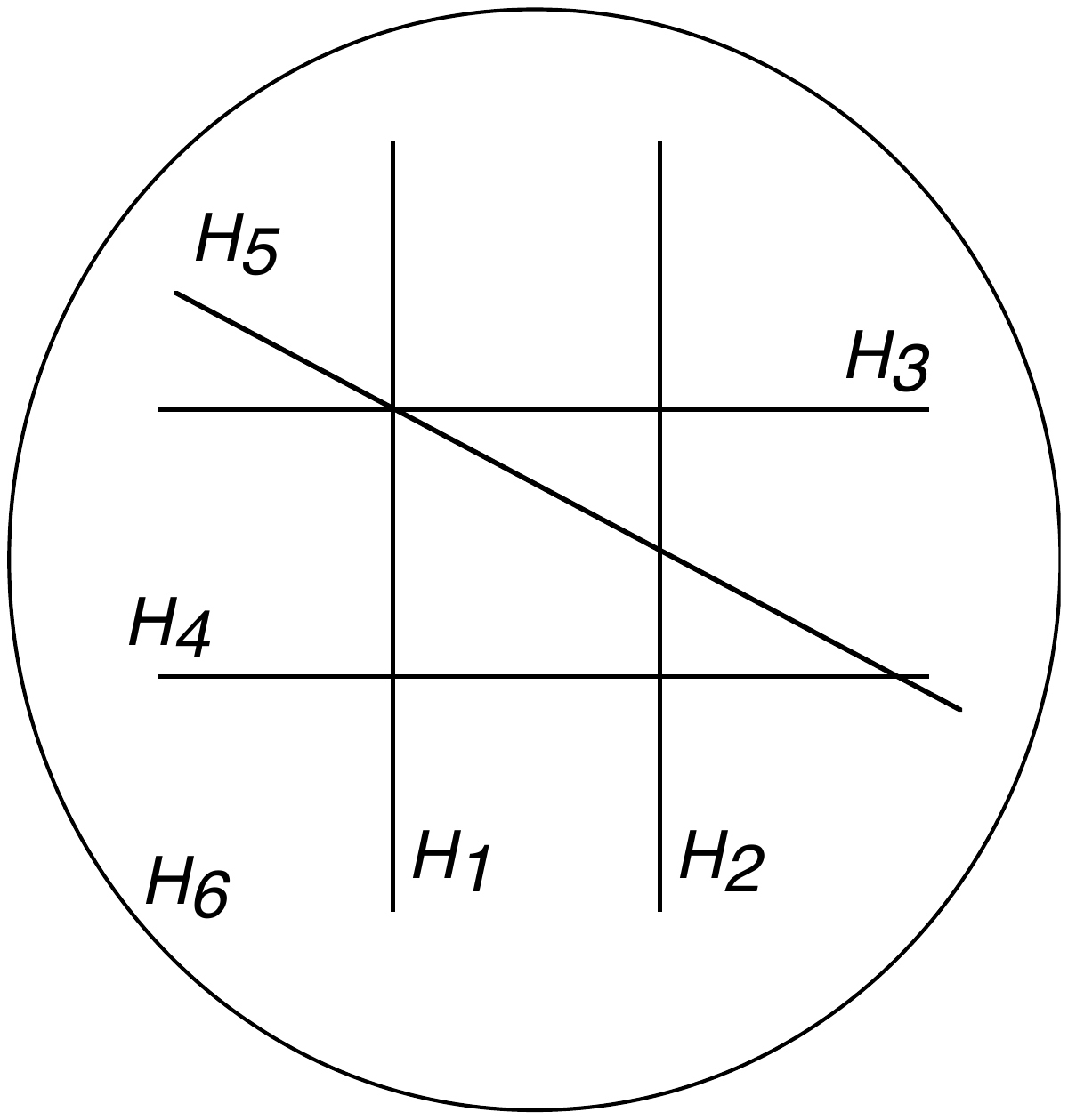}
\caption{\A}
\end{subfigure}
\begin{subfigure}[t]{.49 \textwidth}
\centering
\includegraphics[scale=.2]{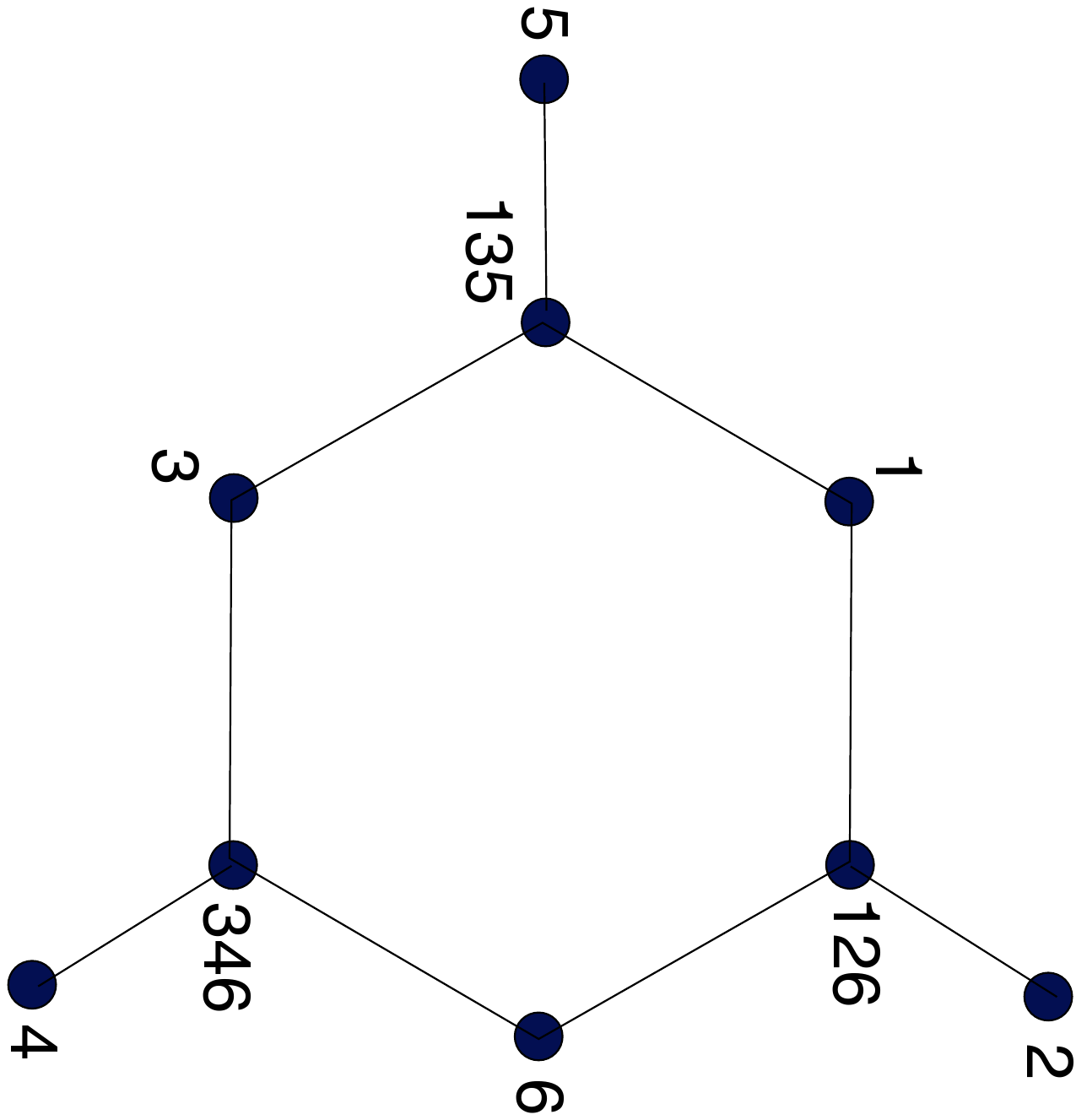}
\caption{$\Lambda_\X$}
\end{subfigure}
\caption{Example \ref{ex:arvola2}}
\label{fig:arvola2}
\end{figure}

We have $\bar{A}\cong F_2\times F_2\times F_2$, while $\bar{A}/\image_\X\cong \Z$, since it is free abelian of rank $(3+3+3) -6-3+1=1$. 
One can choose the generators of $F_2\times F_2 \times F_2$ so that each maps to the same generator of the quotient, see Proposition~\ref{prop:bestbr} for a more general result.  Then $\image_\X=\bar{\rho}_\X(\bar{G})$ is isomorphic to the Stallings group \cite{Stall63}, the kernel of the map $F_2\times F_2\times F_2 \to \Z$ sending every canonical generator to 1. 

The standard set of generators for the deconed arrangement in Example~\ref{ex:arvola} can be extended to a standard set of generators $\{y_1,y_2,y_3,y_4,y_5,y_6\}$ of $G$, with $y_6$ corresponding to the line at infinity, and $y_1\cdots y_6$ central. Then the restriction $\widehat{\rho}_\X \colon \widehat{G} \to \widehat{A}$ is exactly the homomorphism analyzed in that example, where it was shown to be injective. Then $\bar{\rho}_\X$ is injective by Corollary~\ref{cor:affinj}, so in fact $\bar{G}$ is isomorphic to the Stallings group. This was first observed by D.~Matei and A.~Suciu \cite{MatSuc}.
\label{ex:arvola2}
\end{example}

\subsection{Qualitative properties of arrangement groups}
Now we combine the descriptions of the kernel and cokernel of $\rho_\X$ to draw some conclusions about arrangement groups.
We record some properties of $\image_\X$, immediate from the definition and from properties of free groups. Recall that a discrete group has the Haagerup property, or is a-T-menable, if it acts properly and isometrically on an affine Hilbert space -- see \cite{CCJJV01}. Free groups have the Haagerup property, as do subgroups and finite direct products of groups with the Haagerup property.

\begin{theorem} Let 
$\X\subset 2^\A$ be a set of rank two flats of a central arrangement $\A$, 
$\bar{\rho}_\X\colon \bar{G} \to \prod_{S\in\X} \bar{G}_S$  and $\image_\X=\im \bar\rho_\X$ as above. Then 
\begin{enumerate}
\item $\image_\X$ is residually free.
\item $\image_\X$ is torsion-free.
\item $\image_\X$ is residually torsion-free nilpotent.
\item $\image_\X$ has solvable word and conjugacy problems.
\item $\image_\X$ has a faithful linear representation.
\item $\image_\X$ is residually finite.
\item $\image_\X$ has the Haagerup property.
\end{enumerate}
If $\bar{\rho}_\X$ is injective, then $\bar{G}$ has these properties.
\label{thm:properties}
\end{theorem}
\begin{proof}
The direct product of finitely generated free groups $\prod_{S\in \X} \bar{G}_S$ has all these properties, and $\image_\X$ is a subgroup. 
\end{proof}

The results of \cite{MMV98} allow us to determine the homological finiteness type of $\image_\X$ (and hence of $\bar{G}$ if $\rho_\X$ is injective). Recall that a group $G$ is of type $F_k$ if there is a $K(G,1)$ with finite $k$-skeleton. Let $\Gamma$ be the graph associated with the right-angled Artin group $\bar{A}$. Then $\Gamma$ is isomorphic to the complete multipartite graph with parts $\Gamma_i$ of cardinality $k_i=|S_i|-1$ for $1 \leq i \leq m$. 
The vertices of $\Gamma$ correspond to certain of the generators $b_{ij}$, which again correspond to edges of $\Lambda$. For each $i$ choose $y_{j_i} \in S_i$, and let $S_i'=S_i-\{y_{j_i}\}$. Then the vertices of $\Gamma$, that is, the generators in a RAAG presentation of $\bar{A}$, can be taken to be $\{b_{ij} \mid 1 \leq i \leq m, y_j \in S_i'\}$. Let us say such a vertex $b_{ij}$ of $\Gamma$, or the corresponding edge of $\Lambda_\X$, is {\em living} if it has nontrivial image in $\bar{A}/\image_\X$.

\begin{proposition} An edge of $\Lambda_\X$ is living if and only if it is not an isthmus of $\Lambda_\X$.
\label{living}
\end{proposition}

\begin{proof} Having identified $\bar{A}/\image_\X$ with the first cohomology group $H^1(\Lambda_\X,\Z)$, one sees that the living edges of $\Lambda_\X$ are those that appear in cycles in $\Lambda_\X$, which are, by definition, the non-isthmuses of $\Lambda_\X$. 
\end{proof}

We require the notion of $k$-acyclic-dominating subcomplex from \cite{MMV98}. The definition is recursive. Recall a complex $\KK$ is $k$-acyclic if $\tilde{H}_i(\KK,\Z)=0$ for $0 \leq i \leq k$.

\begin{definition} Let $\LLL$ be a simplicial complex. A subcomplex $\KK$ of the complex $\LLL$ is $(-1)$-acyclic-dominating if $\KK$ is nonempty. For $k \geq 0$, a subcomplex $\KK$ of the complex $\LLL$ is $k$-acyclic-dominating if for every vertex $v$ of $\LLL-\KK$, $\lk_\LLL(v) \cap \KK$ is $(k-1)$-acyclic and $(k-1)$-acyclic-dominating in $\lk_\LLL(v)$.
\end{definition}

\begin{lemma} Let $m \geq 2$ and $\LLL=\LLL_1 \ast \cdots \ast \LLL_m$ where $\LLL_i$ is a nonempty 0-dimensional complex for $1 \leq i \leq m$. Let $\KK_i$ be a nonempty subset of $\LLL_i$ for $1 \leq i \leq m$, and let $\KK=\KK_1 \ast \cdots \ast \KK_m$. Then $\KK$ is $(m-2)$-acyclic and is $(m-2)$-acyclic-dominating in $\LLL$.
\label{dom}
\end{lemma}

\begin{proof} The first statement holds because $\KK$ is a join of $m$ nonempty $0$-dimensional complexes, hence is a bouquet of $(m-1)$-spheres. For the second statement, induct on $m$. Note, if $v \in \LLL_i$ then $\lk_\LLL(v)=\bigast_{r \neq i} \LLL_r$ is a bouquet of $(m-2)$-spheres, and $\lk_\LLL(v) \cap \KK=\bigast_{r \neq i} \KK_r$ is as well, by the assumption that $\KK_r$ is nonempty for each $r$. Then, if $m=2$, $\lk_\LLL(v) \cap \KK$ is nonempty, so $\KK$ is $0$-acyclic-dominating in $\LLL$, and, if $m \geq 3$, then $\lk_\LLL(v) \cap \KK$ is $(m-3)$-acyclic, and is $(m-3)$-acyclic-dominating in $\lk_\LLL(v)$ by the inductive hypothesis. Then $\KK$ is $(m-2)$-acyclic-dominating in $\LLL$, completing the induction. 
\end{proof}

Let $\LLL=\Fl(\Gamma)$ denote the flag complex of $\Gamma$, the simplicial complex whose $p$-simplices are the cliques of size $p+1$ in $\Gamma$. Since $\Gamma$ is a complete multipartite graph, $\LLL$ is a join of $m$ zero-dimensional complexes $\LLL_i=\Fl(\Gamma_i)$, $1 \leq i \leq m$. Let $\KK=\KK(\Gamma)$ be the full subcomplex of $\Gamma$ on the set of living vertices of $\Gamma$. Then $\KK=\KK_1 \ast \cdots \ast \KK_m$ where $\KK_i=\KK \cap \LLL_i$.

\begin{theorem} Suppose $\Lambda_\X$ is connected. Then $\image_\X$ is of type $F_{m-1}$ and not of type $F_m$.
\label{type}
\end{theorem}

\begin{proof} We apply the main theorem of \cite{MMV98}, which in our setting states that the kernel 
$\image_\X$ of the map $\bar{A} \to \bar{A}/\image_\X$ has type $F_k$ if and only if $\KK$ is $(k-1)$-acyclic and is a $(k-1)$-acyclic-dominating subcomplex of $\LLL$. 
Since $\Lambda_\X$ is connected, one can choose $y_{j_i} \in S_i$ so that $S_i-\{y_{i_j}\}$ contains a generator of $A$ corresponding to a non-isthmus.
Proposition~\ref{living} and Lemma~\ref{dom} then imply that $\image_\X$ is of type $F_{m-1}$. Moreover, $\KK$ is not $(m-1)$-connected, so $\image_\X$ is not of type $F_m$. 
\end{proof}

\begin{corollary} Suppose $G=G(\A)$ for a central arrangement \A, and \X\ is a set of $m \geq 2$ rank two flats 
such that $\rho_\X$ is injective, and $\Lambda_\X$ is connected. 
Then $M(\A)$ is not aspherical.
\end{corollary}

\begin{proof} First, $M(\A)$ is aspherical if and only if $\bar{M}(\A)$ is. The projective complement $\bar{M}(\A)$ has the homotopy type of a finite complex, which provides a finite $K(\bar{G},1)$ if $\bar{M}(\A)$ is aspherical, which then implies $\bar{G}$ is of type $F_m$ for all $m \geq 0$. 
\end{proof}

\begin{example} 
For the arrangement $\A$ of Example~\ref{ex:arvola2}, identify 
$\bar{A}\cong F_2 \times F_2 \times F_2$ with the subgroup $\widehat{A}$ of $A$. Let $S_1=\{y_1,y_2,y_6\}, S_2=\{y_1,y_3,y_5\},S_3=\{y_3,y_4,y_6\}$, and $\X=\{S_1,S_2,S_3\}$. Choose $y_{j_1}=y_2$, $y_{j_2}=y_5$, and $y_{j_3}=y_4$. Then
$\Lambda_\X$ is connected, 
and $\image_\X$ is of type $F_2$ (that is, $\bar{G}$ is finitely-presented) but not of type $F_3$. Hence $\bar{G}$ is of type $F_2$ but not of type $F_3$. This was first established by different methods in unpublished work of Arvola \cite{Arv92}, which motivated Matei and Suciu's identification of this group with the Stallings group. The details of their computation appear in \cite{Suc14}.
\label{ex:arvola3}
\end{example}

\begin{example} \label{ex:4not5}
Let $\A=\{H_1,\dots,H_7\}$ be the arrangement in $\C^3$ obtained by coning the arrangement of Example~\ref{ex:kohno}. The set of generators of that example can be extended to a standard set of generators $\{y_1,\ldots, y_7\}$ of $G$ with $y_1\cdots y_7$ central in $G$. Let 
\[
\begin{aligned}
\X&=\{S_1,\ldots,S_5\}\\
&=\{\{H_1,H_3,H_5\},\{H_2,H_4,H_6\},\{H_1,H_2,H_7\},\{H_3,H_4,H_7\},
\{H_5,H_6,H_7\}\}.
\end{aligned}
\]
By Example~\ref{ex:kohno} and Corollary~\ref{cor:affinj}, $\bar{\rho}_\X$ is injective. Thus, $\bar{G}$ 
satisfies properties (i)--(vii) of Theorem \ref{thm:properties}. 
The graph $\Lambda_\X$ is connected, so 
$\bar{G}$ is of type $F_4$ but not $F_5$ by Theorem~\ref{type}.
\end{example}

\begin{example}
Similar calculations apply to two (combinatorially distinct) seven-line arrangements which appear in \cite{Fa97}, illustrated in Figure~\ref{fig:seven}. In these cases the groups are of type $F_3$ but not $F_4$.

\begin{figure}[h]
\begin{subfigure}{.49 \textwidth}
\centering
\includegraphics[scale=.2]{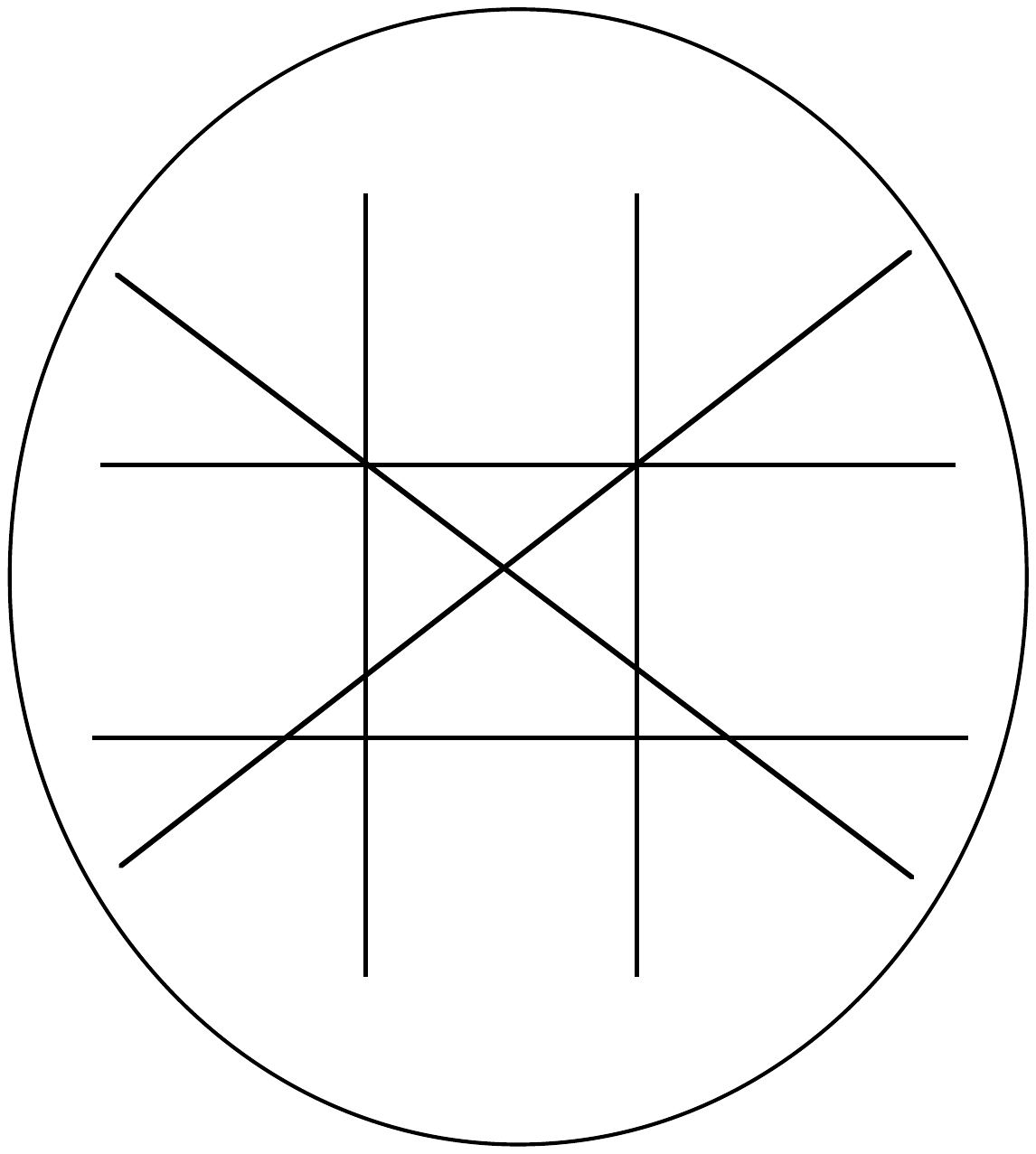}
\end{subfigure}
\begin{subfigure}{.49 \textwidth}
\centering
{\includegraphics[scale=.2]{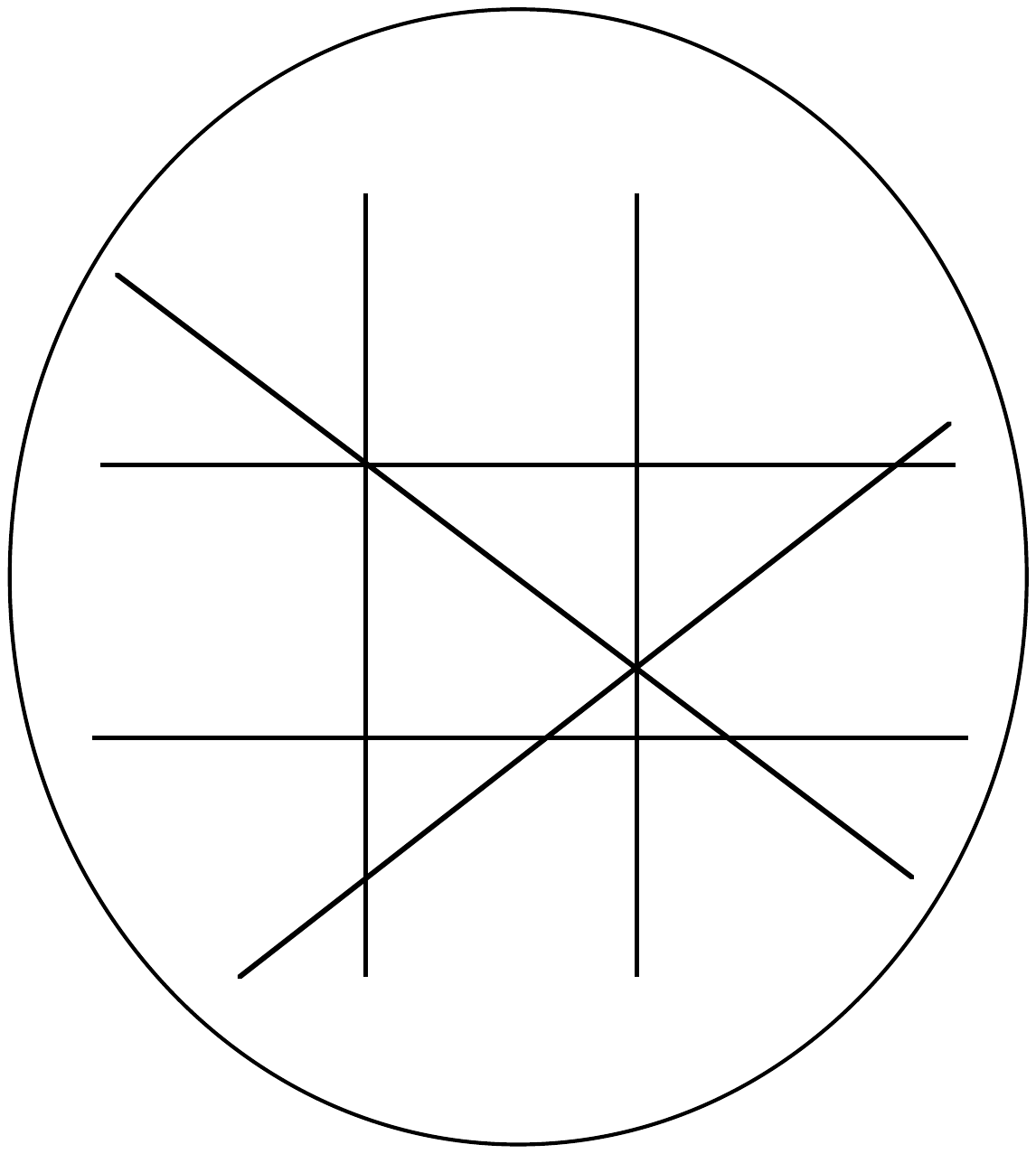}}
\end{subfigure}
\caption{Arrangements with groups of type $F_3$ but not $F_4$}
\label{fig:seven}
\end{figure}
\end{example}

\subsection{Bestvina-Brady arrangement groups}
In \cite{ACM15}, Artal, Cogolludo, and Matei generalize Example~\ref{ex:arvola3}.
For positive integers $n_1, \ldots, n_r$, they construct a projective line arrangement $\A(n_1,\ldots, n_r)$ of $\sum_{i=1}^r n_i $ lines, forming an $r$-gon of points of multiplicities $n_1+1, \ldots, n_r+1$. More precisely, there are $r$ distinguished lines in general position, forming the sides of an $r$-gon, with vertices labelled $1, \ldots, r$ in consecutive order. Passing through the $i^{\rm th}$ vertex there is a ``bushel" of $n_i-2$ lines, pairwise intersecting in double points elsewhere. In the language of \cite{EGT09}, the graph of multiple points forms an $r$-cycle.

We have illustrated this ``arrangement schema" in Figure~\ref{fig:acm}, with each bushel of lines represented by a single colored line. For later purposes we have placed one of the distinguished lines at infinity. The underlying matroids of these arrangements have path-connected realization spaces, so the diffeomorphism type of the complement is uniquely determined by the integers $n_1,\ldots, n_r$ by \cite{R1}.

\begin{figure}[h]
\centering
\includegraphics[scale=.25]{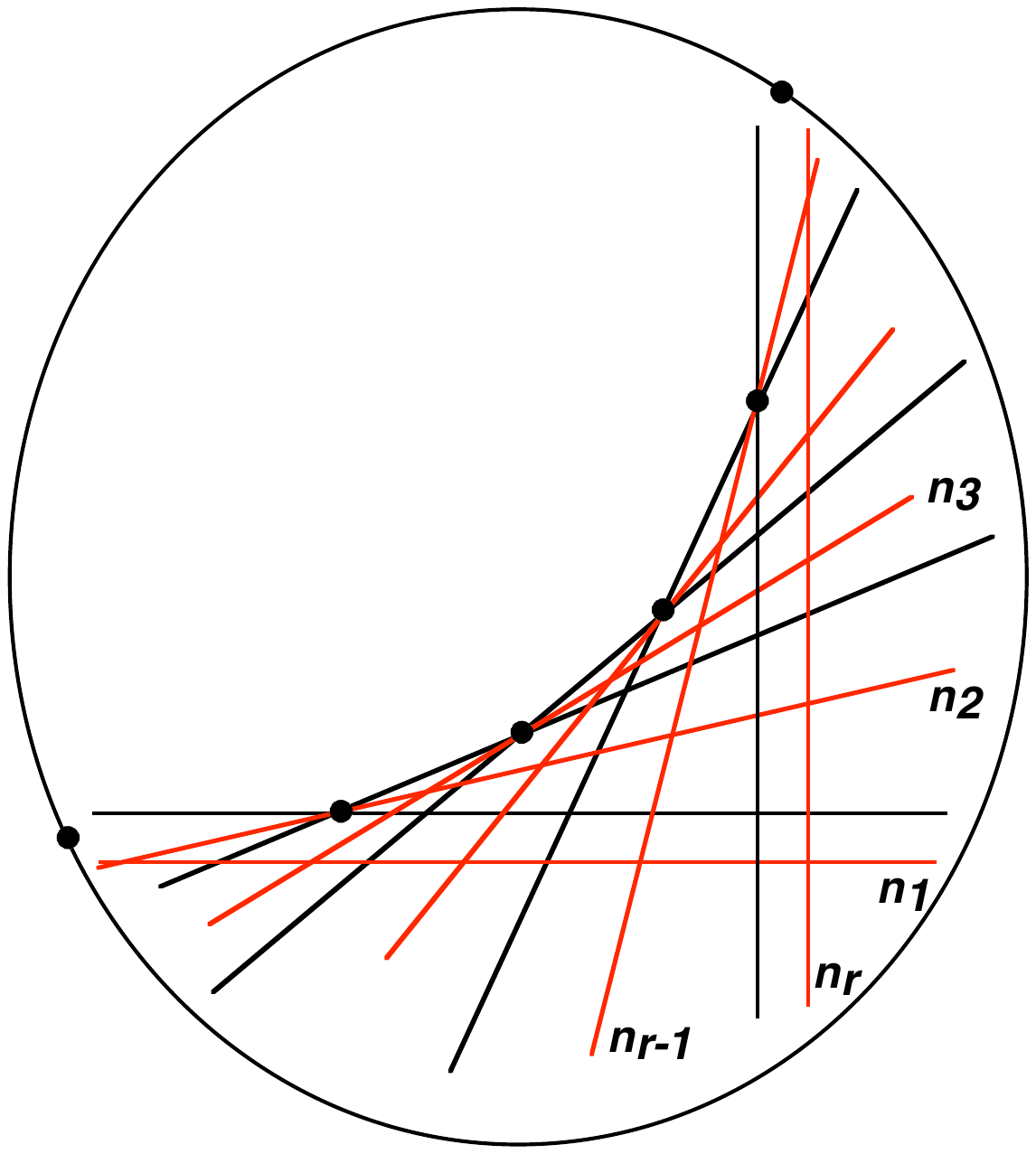}
\caption{The Artal-Cogolludo-Matei arrangement schema}
\label{fig:acm}
\end{figure}

In \cite{ACM15}, it is shown 
that the corresponding arrangement group is isomorphic to the kernel of the map $F_{n_1} \times \cdots \times F_{n_r} \to \Z$ sending each generator to 1, and hence is a Bestvina-Brady group of type $F_{r-1}$ and not $F_r$. We can reproduce their result using our method.

\begin{theorem} Let $\A=\A(n_1,\ldots, n_r)$ as defined above, and $G=G(\A)$. Let \X\ be the set of rank two flats of \A\ of cardinality greater than two, and $\bar{\rho}_\X \colon \bar{G} \to \prod_{S\in \X} \bar{G}_S$. Then $\bar{\rho}_\X$ is injective.
\label{thm:bestbrad}
\end{theorem}

\begin{proof} Let $\widehat{G}$ be the group of a decone $d\A$, relative to one of the sides of the $r$-gon. We first show that $\widehat{G}$ is a cyclically-presented hyperplane (arrangement) group,  so $Y$ is adapted to \X\ by Corollary~\ref{conjfree}. 
Assume that the $y$-axis is far to the right in Figure \ref{fig:acm}, and order the non-vertical lines of $d\A$ by increasing $y$-intercept.  Then order the vertical lines of $d\A$ by increasing $x$-intercept.  Denote the distinguished lines of $d\A$ by $q_1,\dots,q_{r-1}$, $q_i<q_{i+1}$, so line $q_1$ is the top horizontal line, and $q_{r-1}$ is the left-most vertical line.  Write $n=\sum_{i=1}^r n_i$.

Sweeping a line of large negative slope from right to left in Figure \ref{fig:acm}, the Randell algorithm \cite{R3} yields a presentation for $\widehat{G}=G(d\A)$ with relations
\[
y_{q_i} y_{q_i+1} \cdots y_{q_{i+1}-1}y_{q_{i+1}}=y_{q_i+1} \cdots y_{q_{i+1}-1} y_{q_{i+1}} y_{q_i}= \ldots\ldots =y_{q_{i+1}} y_{q_i} y_{q_i+1} \cdots y_{q_{i+1}-1}
\]
for $1 \le i \le r-2$, and commutator relations $[y_s,y_t^{w_{s,t}}]$ where $1\le s < t \le n-1$, $s,t\notin \{q_i,\dots,q_{i+1}\}$ for $1\le i\le r-1$, and $w_{s,t}$ is a word in the $y_k$.  If $t=q_j$ for some $j$ or $t>q_{r-1}$, then $w_{s,t}=1$.  If $q_j < t < q_{j+1}$, then $w_{s,t}=y_{t+1} \cdots y_{q_{j+1}}$ and $s<q_j$.  
For each $j$, $1 \le j \le r-2$, the family of commutator relations
\[
[y_s,y_{q_{j}}], \ [y_s,y_t^{w_{s,t}}],\ [y_s,y_{q_{j+1}}], \quad t=q_j+1,\dots,q_{j+1}-1,
\]
provided by the Randell algorithm is easily seen to be equivalent to the family 
\[
[y_s,y_{q_{j}}], \ [y_s,y_t],\ [y_s,y_{q_{j+1}}], \quad t=q_j+1,\dots,q_{j+1}-1.
\]
Thus, $\widehat{G}$ admits a conjugation-free presentation. This may also be seen by applying the main result of \cite{EGT09}, since the graph of multiple points has a unique cycle.

We show $\widehat{\rho}_\X\colon \widehat{G} \to \prod_{S\in \X} G_S$ is injective using Theorem~\ref{thm:injtest} and Proposition~\ref{prop:simplify}. The result will then follow from Corollary~\ref{cor:affinj}. 
The first condition of Theorem~\ref{thm:injtest} holds by the Randell algorithm simplifications given in the previous paragraph. 
For the decone $d\A$, let $\X=\X_0 \cup \X_\infty$, where
$\X_0=\{ \{q_i,\dots,q_{i+1}\} \mid 1 \le i \le r-2\}$ and $\X_\infty=\{ \{1,\dots,q_1\}, \{q_{r-1},\dots,n-1\}\}$.  For $T \in \X$ and $j \notin T$, we must show that $[y_j, [G_T,G_T]]=1$.  There is a unique $S \neq T$ in $\X$ with $j \in S$.  If either $T \in \X_\infty$ or $S \cap T=\emptyset$, then clearly 
$[y_j, [G_T,G_T]]=1$ since $[y_j,y_t]=1$ for all $t \in T$ in either of these instances.  It remains to consider the case where $T\in \X_0$ and $S \cap T \neq\emptyset$.  In this case, $S$ and $T$ satisfy the hypotheses of Proposition \ref{prop:simplify}.  The result follows. 
\end{proof}

\begin{corollary} If $\A=\A(n_1,\ldots, n_r)$, the group $\bar{G}(\A)$ satisfies properties (i)-(vii) of Theorem~\ref{thm:properties}.
\end{corollary}

\begin{proposition} \label{prop:bestbr}
The injection $\bar{\rho}_\X$ realizes $\bar{G}$ as a Bestvina-Brady group.
\end{proposition}

\begin{proof}
The graph $\Lambda_\X$ consists of a cycle formed by the distinguished hyperplanes and the elements of \X\, together with a pendant edge incident with each of the other vertices. Then $\coker(\bar{\rho}_\X)\cong H^1(\Lambda_\X,\Z)\cong \Z$.  

To complete the proof, we exhibit a free basis of $\bar{A}=\prod_{i=1}^r\bar{G}_{S_i}$ whose elements all map to the same generator of $\coker(\bar{\rho}_\X)$. Fix $i$ and write $S_i=\{H_1,\ldots, H_{n_i}\}$ with $H_1$ and $H_2$ being the distinguished hyperplanes, and let $y_{i1}, \ldots, y_{in_i}$ be the corresponding canonical generators of $G_{S_i}$. Each of the $y_{ij}$ for $j\geq 3$ lies in $\image_\X$, hence maps to zero in $\coker(\bar{\rho}_\X)$. Using the identification of $\coker(\bar{\rho}_\X)$ with $\coker((\bar{\rho}_\X)_{\ab})$, it is clear that $y_{i1}$ and $y_{i2}$ map to $\pm 1$. By reversing orientation we may assume they each map to 1. Then $\{y_{i1}, y_{i2}, y_{i1}y_{i3},\ldots,  y_{i1}y_{i,n_{i-1}}\}$ is a free basis of $\bar{G}_{S_i}$, each element of which maps to 1 in \Z. Repeating the process for each $i, 1\leq i \leq r$, yields the desired free basis of $\bar{A}$. Since $\bar{G}\cong \image_\X = \ker(\bar{A} \to \coker(\bar{\rho}_\X))\cong \Z$, this proves the claim.
\end{proof}

As noted above, this proposition reproduces results of 
\cite{ACM15}. By \cite{DPS08a}, these are the only quasi-projective groups that are Bestvina-Brady groups, aside from products of free groups. 

\subsection{Decomposable arrangements} 
\label{subsec:decomp}
For any group $G$, let $\{G^n \mid n \geq 1\}$ denote the lower central series of $G$, and let $\phi_n(G)$ denote the rank of the finitely-generated abelian group $G^n/G^{n+1}$. 
Let $\Lie_n(G)=G^n/G^{n+1}$ for $n \geq 1$, and $\Lie(G)=\oplus_{n=1}^\infty \Lie_n(G)$. 
If $Y$ is a set of generators of $G$ and  $\sS \subseteq 2^Y$, we have the product of canonical homomorphisms $\Lie(\rho_\sS)_n\colon \Lie_n(G) \to \prod_{S\in \sS} \Lie_n(G_S)$ for each $n\geq 1$. 	

The following notion is motivated by arrangements theory, as will be seen explicitly below.
\begin{definition} The group $G$ is {\em decomposable} with respect to $Y$ and \sS\ if $\Lie(\rho_\sS)_n$ is an injection for every $n \geq 1$.
\end{definition}
 
The {\em nilpotent residue} $G^\omega$ of $G$ is the 
intersection $\bigcap_{n=1}^\infty G^n$, so $G$ is residually nilpotent if $G^\omega=1$. 

\begin{proposition} 
Let $Y$ be a set of generators for $G$. Assume that $\sS \subseteq 2^Y$ covers $Y$, and that $G$ is decomposable with respect to $Y$ and \sS. If $G_S$ is residually nilpotent for every $S \in \sS$, then the kernel of $\rho_\sS \colon G \to \prod_{S \in \sS} G_S$ is equal to $G^\omega$.
\label{residue}
\end{proposition}

\begin{proof} The product of residually nilpotent groups is residually nilpotent, so $G^\omega\subseteq \ker(\rho_{\sS})$ by the assumption on the $G_S$. Suppose $g\in \ker(\rho_{\sS})- G^\omega$. Choose $n\geq 1$ minimal with $g\not \in G^{n+1}$. Then $g \in G_n$ and $gG_{n+1}$ is a nontrivial element of the kernel of the map $\Lie(\rho_\sS)_n\colon \Lie_n(G) \to \prod_{S\in \sS} \Lie_n(G_S)$, contradicting the assumption that $G$ is decomposable with respect to $Y$ and \sS. 
\end{proof}

Let \A\ be an arrangement and let $\X=\L^{[2]}$ be the set of all rank two flats of \A. Let $G=G(\A)$ and $G_X=G(\A_X)$ for $\A_X \in \X$ denote the corresponding arrangement groups.

\begin{definition}(\cite{PS06a}) The arrangement \A\ is {\em decomposable} if 
\[
\Lie(\rho_\X)_3 \otimes \k \colon \Lie_3(G) \otimes \k \to \prod_{S\in \X} \Lie_3(G_S) \otimes \k 
\]
is an isomorphism for every field \k. 
\end{definition}

If \A\ is decomposable, then $\phi_3(G)=\sum_{S \in \X} \phi_3(G_S)$; this equality depends only on the lattice of \A.  The converse depends on the torsion-freeness of $\Lie_3(G)$, which is presently not known -- see \cite{PS19} in this volume for recent progress. 

Theorem 2.4 of \cite{PS06a} implies $\Lie_n(\rho_\sS)$ is an isomorphism for all $n \geq 2$ for decomposable arrangements.   The arrangements of Examples \ref{ex:arvola2} and \ref{ex:4not5} are decomposable -- see \cite[Section 7]{PS06a}.  
This also follows from our calculations, with the following observation.
\begin{proposition} Suppose \sS\ is a family of rank-two flats that includes all flats of multiplicity greater than two, and $\rho_\sS$ is injective. Then $G$ is decomposable.
\end{proposition}

\begin{proof} Let $A=\prod_{S\in \sS} G_S$. The  hypothesis on \sS\ implies $\Lie_3(A)=\prod_{S\in \X} \Lie_3(G_S)$. By the injectivity hypothesis and Propositions~\ref{cor:normal} and \ref{prop:coker}, there is an exact sequence 
\[
\xymatrix{
1 \ar[r] &G \ar[r]^{\ \rho_\sS} & A \ar[r] & Q \ar[r] &1\\
}
\]
with $Q$ abelian. Then $\rho_\sS$ induces an isomorphism $G^n \to A^n$ for all $n\geq  2$, and hence an isomorphism $\Lie_3(G) \to \Lie_3(A)$.
\end{proof}

\begin{proposition} Suppose \A\ is a decomposable arrangement, and $Y$ is a standard set of generators for the arrangement group $G$. Then $G$ is decomposable with respect to $Y$ and $\X=\L^{[2]}$.
\label{decomparr}
\end{proposition}

\begin{proof} Since $G_X$ is a product of free groups, $\Lie_n(G_X)$ is free abelian for every $n \geq 1$ and $\A_X \in \X$. By \cite{PS06a}, the hypothesis implies $\Lie(\rho_\X)_n$ is an isomorphism for every $n \geq 2$. 
\end{proof}

Propositions ~\ref{residue} and \ref{decomparr} yield the following corollary.

\begin{corollary} Suppose \A\ is a decomposable arrangement with group $G$, and \X\ is a family of rank two flats that covers \A. Then the kernel of $\rho_\X$ is equal to $G^\omega$.
\end{corollary}

We deduce a dichotomy for decomposable arrangement groups. Loosely speaking, decomposable arrangement groups either have conjugation-free presentations, or are not residually nilpotent.

\begin{proposition} Suppose \A\ is a decomposable arrangement with group $G$. Suppose $G$ is residually nilpotent. Then any standard set of generators arising from a braid monodromy presentation of $G$ is adapted to the set $\L^{[2]}$ of all rank two flats.
\label{dicho}
\end{proposition}

\begin{proof} 
Let $Y$ be a standard set of generators arising from a braid monodromy presentation of $G$ \cite{CS2}. Let $\X=\L^{[2]}$. Then \X\ covers \A, so $\rho_\X$ injective, by the preceding corollary and the hypothesis on $G$.  Suppose $\A_X$ is a rank two flat and the subset $S_X$ of $Y$ is not retractive. From the braid monodromy presentation, the quotient $G_X$ has a presentation with generators $y_H$, $H \in \A_X$ and relations 
\begin{equation}
\Bigl[\prod_{H \in \A_X} y_H, y_K\Bigr]=1, \ \ K \in \A_X,
\tag{$\ddag$}
\label{loccom}
\end{equation}
with some fixed order of factors in the product. Since $S_X$ is not retractive, at least one of these relations must fail to hold in $G$. Choose $K \in \A_X$ with $\xi=[\prod_{H \in \A_X} y_H, y_K] \neq 1$ in $G$. We claim $\rho_\X(\xi)=1$, a contradiction. First $\rho_X(\xi)=1$ because the relations \eqref{loccom}  hold in $G_X$. Let $Z \in \X$ with $Z \neq X$. If $K \not \in \A_Z$ $\rho_Z(y_K)=1$ so $\rho_Z(\xi)=1$. Otherwise, since $X$ and $Z$ are rank two flats, $\A_X \cap \A_Z=\{K\}$, so $\rho_Z(\prod_{H \in \A_X} y_H)=y_K$, and $\rho_Z(\xi)=[y_K,y_K]=1$. Then $\rho_\X(\xi)=1$. This completes the proof. 
\end{proof}

\begin{example} The arrangement \A\ given in Example~\ref{ex:kohno} is decomposable, 
and the group $G(\A)$ is residually nilpotent by Theorem~\ref{thm:properties}. Then Proposition~\ref{dicho} implies that the relators $[y_4,y_5^{y_1}]$ and $[y_1,y_6^{y_2}]$ in the Randell presentation of $G(\A)$ can be replaced by $[y_4,y_5]$, and $[y_1,y_6]$, respectively.
\end{example}

\section*{Acknowledgements}  We are grateful to Graham Denham, Daniel Matei, and Alex Suciu for many helpful discussions. Theorem 4.1.6 was formulated with the help of Daniel Malcolm when he was an undergraduate research student of the second author. 

This project was begun during the program on hyperplane arrangements at Mathematical Sciences Research Institute in 2004, and was continued during conferences at the Pacific Institute of Mathematical Sciences, the Fields Institute, the University of Zaragoza, and Hokkaido University, and during the program on configuration spaces at Centro di Ricerca Matematica Ennio de Giorgi in Pisa. The authors thank these institutions and the meeting organizers for their support and hospitality. We also thank several anonymous referees for their careful reading and helpful suggestions.

\end{document}